\documentclass[11pt]{article}

\usepackage[utf8]{inputenc}
\usepackage[T1]{fontenc}
\usepackage{graphicx}
\usepackage{url}
\usepackage{hyperref}
\usepackage{amsfonts}
\usepackage{amsmath}
\usepackage{amssymb}
\usepackage{amsthm}
\usepackage{appendix}
\usepackage{booktabs}
\usepackage{listings}
\usepackage{xcolor}
\usepackage{enumitem}
\usepackage[margin=3.1cm]{geometry}
\usepackage{caption}
\hypersetup{colorlinks=true, linkcolor=blue!50!black, citecolor=green!40!black,
            urlcolor=blue!60!black}

\theoremstyle{plain}
\newtheorem{theorem}{Theorem}[section]
\newtheorem{proposition}[theorem]{Proposition}
\newtheorem{lemma}[theorem]{Lemma}
\newtheorem{corollary}[theorem]{Corollary}
\theoremstyle{definition}
\newtheorem{definition}[theorem]{Definition}
\newtheorem{assumption}[theorem]{Assumption}

\newtheorem{algorithm}[theorem]{Algorithm}
\theoremstyle{remark}
\newtheorem{remark}[theorem]{Remark}

\newcommand{\Rbb}{\mathbb{R}}

\newcommand{\Om}{\Omega}
\newcommand{\Th}{\mathcal{T}_h}
\newcommand{\Vh}{V_h}
\newcommand{\dive}{\operatorname{div}}
\newcommand{\grad}{\nabla}
\newcommand{\Hdiv}{\boldsymbol{H}(\dive,\Om)}
\newcommand{\RT}{\boldsymbol{RT}}
\newcommand{\BDM}{\boldsymbol{BDM}}
\newcommand{\sig}{\boldsymbol{\sigma}}
\newcommand{\tauv}{\boldsymbol{\tau}}
\newcommand{\pv}{\boldsymbol{p}}

\newcommand{\nv}{\boldsymbol{n}}
\newcommand{\xv}{\boldsymbol{x}}
\newcommand{\norm}[1]{\lVert #1 \rVert}
\newcommand{\enorm}[1]{\vert\!\vert\!\vert #1 \vert\!\vert\!\vert}
\newcommand{\ip}[2]{\left( #1 , #2 \right)}
\newcommand{\Vfl}{\boldsymbol{V}}
\DeclareMathOperator{\osc}{osc}

\definecolor{lightgray}{gray}{0.95}
\lstdefinestyle{py}{
  language=Python,
  basicstyle=\ttfamily\footnotesize,
  keywordstyle=\color{blue!60!black},
  commentstyle=\color{green!40!black}\itshape,
  stringstyle=\color{red!60!black},
  backgroundcolor=\color{lightgray},
  frame=lines, framesep=2mm,
  numbers=left, numberstyle=\tiny\color{gray},
  breaklines=true, showstringspaces=false,
  rangeprefix=\#\ , includerangemarker=false,
}

\newif\ifrevisions
\revisionsfalse
\definecolor{dorado}{RGB}{218,165,32}
\definecolor{mygreen}{RGB}{30, 200, 30}
\ifrevisions
  \newcommand{\franz}[1]{\textcolor{dorado}{#1}}
  \newcommand{\zuqi}[1]{\textcolor{red}{#1}}
  \newcommand{\michel}[1]{\textcolor{blue}{#1}}
  \newcommand{\claudio}[1]{\textcolor{mygreen}{#1}}
\else
  \newcommand{\franz}[1]{#1}
  \newcommand{\zuqi}[1]{#1}
  \newcommand{\michel}[1]{#1}
  \newcommand{\claudio}[1]{#1}
\fi

\title{A posteriori error estimation for coarse or distorted meshes}

\author{
Franz Chouly\thanks{Centro de Matem\'atica (CMAT), Universidad de la Rep\'ublica,
  IRL CNRS 2030 IFUMI, 11400 Montevideo, Uruguay,
  \url{fchouly@cmat.edu.uy}}
\and Michel Duprez\thanks{Universit\'e de Strasbourg, CNRS, Inria, ICube, UMR 7357,
  \'equipe MIMESIS, F-67000 Strasbourg, France,
  \url{michel.duprez@inria.fr}}
\and Claudio Lobos\thanks{Departamento de Inform\'atica, Universidad T\'ecnica
  Federico Santa Mar\'ia, Avenida Vicu\~na Mackenna 3939, Santiago, Chile,
  \url{claudio.lobos@usm.cl}}
\and Zuqi Tang\thanks{Univ.\ Lille, Arts et M\'etiers Institute of Technology,
  Centrale Lille, Junia, ULR 2697 -- L2EP, F-59000 Lille, France,
  \url{zuqi.tang@univ-lille.fr}}}

\date{}

\begin{document}

\maketitle

\begin{abstract}
%
Finite element simulations carried out on meshes issued from real-world
geometries --- patient-specific anatomies, scanned industrial parts, fractured
media --- almost never satisfy the shape-regularity assumptions under which
classical a posteriori error analysis is written. Practitioners are therefore
left without a quantitative answer to a very legitimate question: \emph{how
wrong is my simulation when my mesh is coarse and contains badly shaped
elements?}

\franz{In this work, we study to which extent} equilibrated flux estimators
\franz{can }provide such an answer. \franz{Indeed,} they deliver
a guaranteed and fully computable upper bound of the energy-norm discretization
error, with a constant equal to one; and, crucially for the present study, its
derivation uses no shape-regularity, quasi-uniformity or angle condition. The
only geometric ingredient is the Poincar\'e constant on a single cell, which
Payne and Weinberger showed to be at most $h_K/\pi$ on \emph{any}
convex cell $K$, however distorted (\franz{$h_K$ being the cell size}).
Mesh quality affects the local efficiency --- the sharpness of the
estimator --- but never its reliability, so that the failure mode under
distortion is pessimism rather than false confidence.

\franz{We study the case of conforming \michel{Lagrange finite element}
methods\michel{: the analysis is written for general cells in two and three
space dimensions, and the numerical experiments are carried out} with
simplicial \michel{elements} in two dimensions.} We detail two flux reconstructions --- a global
$H(\dive)$ mixed one and the patchwise one (in the spirit of works from Braess and Sch\"oberl or
Ern and Vohral\'ik) --- and we map the geometric quality measures used by the
meshing community onto the constants that do, and do not, appear in the
analysis. The estimator is
then assessed numerically on families of coarse, locally damaged, anisotropic
and globally distorted meshes, against the residual, Zienkiewicz--Zhu\michel{-type
recovery} and
Bank--Weiser estimators, and a duality argument extends the guarantee to a
quantity of interest.

Across the $49$ computations reported here --- including meshes of a
letter glyph produced by two independent meshers, one of which ships with a
zero-area cell --- the guaranteed bound never failed, while the
\michel{recovery} estimator reported \emph{less} error than was committed in
every run on a mesh of the unit square. The Bank--Weiser estimator is sharper
on every well-shaped mesh, but a single needle cell out of two thousand costs
it a factor of $2.6$, and it falls below the true error on the L-shaped domain,
where nothing in its output signals the failure. The two reconstructions of the
guaranteed bound are found to degrade under \emph{different} distortions: the
global one under global anisotropy, the patchwise one under an accumulation of
flattened cells within a single patch.
\end{abstract}





\section{Introduction}
\label{sec:intro}


In an idealised numerical analysis paper, the mesh is a parameter of the method:
it is refined at will, it is shape-regular, and asymptotic statements in \zuqi{mesh size} $h$ are
meaningful. In a large part of \michel{engineering and biomedical practice}, the
situation is reversed: the mesh is a \emph{datum}. It is produced by an automatic
mesher from a segmented medical image, from a CAD model with slivers, or from a
geometry containing thin features and cracks; it is as coarse as the computing
budget allows; and it contains a small (or not so small) number of elements that
any quality criterion would reject. Remeshing may be impossible (the geometry is
not available any more), or undesirable (the mesh carries material data, or is
the output of a registration procedure as in, \emph{e.g.}, \cite{bucki2010,lobos2013}).


The question addressed here is therefore not \emph{how to build a better mesh},
but rather: \emph{given this mesh and this solution, what is the guaranteed
accuracy of the computed quantity?} Any answer that involves an unspecified
generic constant $C$ --- as is the case for the classical explicit residual
estimator --- is of little practical use, and any answer whose validity is
conditioned on a shape-regularity parameter is precisely useless in the regime
we target, since that parameter blows up. What is wanted is a number that is
computable from the data at hand, that is an upper bound with no hidden
constant, and whose derivation does not quietly assume the mesh to be good.


\paragraph{A posteriori error estimation.}
A posteriori error estimation started with
Babu\v{s}ka and Rheinboldt \cite{babuskarheinboldt1978};
the literature is now vast, and we refer to the monographs of Ainsworth and Oden
\cite{ainsworthoden2000}, Verf\"urth \cite{verfurth2013}, Repin
\cite{repin2008}, to the survey of Nochetto, Siebert and Veeser
\cite{nochetto2009} for the adaptive side, and to the lecture notes of
Vohral\'ik \cite{vohralik_notes}, which we follow closely in
Section~\ref{sec:eqflux}. 
  
\franz{Equilibrated \michel{flux}  
  estimators are based on ideas from
        Prager--Synge \cite{pragersynge1947}
and early contributions date back to the pioneering work of          Lad\`eveze
        \cite{ladeveze1983,ladeveze2004}}\michel{; later contributions include} Destuynder--M\'etivet
        \cite{destuynder1999},
        Luce--Wohlmuth \cite{luce2004}, Braess--Sch\"oberl
        \cite{braess2008,braess2009}, Ern--Vohral\'ik
        \cite{ernvohralik2015,ernvohralik2020},
        Carstensen--Merdon \cite{carstensen2013,carstensen2014},
        Becker--Capatina--Luce \cite{becker2016}\michel{, all of which} give a guaranteed bound
        with constant one. Their common root is convex duality and the
        primal--dual gap: see Repin \cite{repin2000,repin2008}, Han
        \cite{han2005}, and, more recently and in a much broader nonsmooth
        setting, Bartels and Kaltenbach \cite{bartels2023}. We come back to this
        reading in Remark~\ref{rem:duality}.
Fully computable bounds of residual type also exist, at the price of a careful
tracking of the interpolation constants: see Veeser and Verf\"urth
\cite{veeser2009} and Carstensen and Merdon
\cite{carstensen2013aspects,carstensen2014}. Those constants still depend on the
element shapes, which is exactly what we want to avoid here.

\paragraph{Finite elements on bad meshes.}
On the a priori side, it is classical since Babu\v{s}ka and Aziz
\cite{babuskaaziz1976} that the \emph{minimum} angle condition can be replaced by
a \emph{maximum} angle condition in 2D, and by suitable variants in 3D
\cite{krizek1992}; anisotropic interpolation theory
\cite{apel1999,formaggia2001,picasso2003} pushes this further. Duprez, Lleras and
Lozinski \cite{duprez2019damage} showed that a \emph{locally damaged} mesh --- a
few arbitrarily distorted cells in an otherwise regular mesh --- does not spoil
the convergence rate, nor irreparably the conditioning. 

\paragraph{Mesh quality measures.}
The meshing community quantifies element quality through geometric or algebraic
measures: aspect ratio, minimum/maximum dihedral angle, (scaled) Jacobian,
condition number of the Jacobian, mean ratio, and their many variants; see the
monograph of Frey and George \cite{freygeorge2008}, the algebraic framework of
Knupp \cite{knupp2001} and the discussion of Shewchuk \cite{shewchuk2002} on which of
those measures actually control interpolation error and conditioning. One
contribution of the present work is to connect these practical measures with the
constants that appear (and, more importantly, with those that do \emph{not}
appear) in the equilibrated flux analysis.

\subsection{Contributions and outline}

\begin{enumerate}[nosep]
  \item We revisit the equilibrated flux framework with an explicit tracking of
        the diffusion parameter $\mu$ and of \emph{every} geometric assumption
        used in the derivation of the upper bound
        (Sections~\ref{sec:setting}--\ref{sec:eqflux}). We isolate the single
        place where geometry enters --- the Poincar\'e constant on a cell,
        Lemma~\ref{lem:poincare} --- and we explain how to keep the bound
        guaranteed on non-convex cells.
  \item We review the geometric quality measures used in meshing and map them
        onto the constants of the analysis (Section~\ref{sec:quality}),
        distinguishing between what affects the \emph{guarantee} (almost
        nothing) and what affects the \emph{efficiency} (a lot). 
  \item We describe two flux reconstructions
        (Sections~\ref{sec:mixed} and \ref{sec:patchwise}): a global mixed 
        \michel{one, in which the flux is sought in a Raviart--Thomas or
        Brezzi--Douglas--Marini subspace of $\Hdiv$ over the whole mesh,}
        expensive but trivial to implement, and the patchwise
        one of Braess--Sch\"oberl / Ern--Vohral\'ik, cheap and parallel, and we
        discuss the practical pitfalls on distorted cells.
        Both are implemented, and Section~\ref{sec:numerics_pw} shows
        that they are not interchangeable under distortion: they lose sharpness
        under \emph{different} geometric defects.
  \item We compare with the residual, Zienkiewicz--Zhu\michel{-type recovery} and Bank--Weiser
        estimators (Section~\ref{sec:comparison}) and quantify how each of them
        degrades with distortion. The comparison with Bank--Weiser
        (Section~\ref{sec:numerics_bw}) is the substantive one: it is the
        sharper estimator on any mesh one would choose, and the more fragile on
        any mesh one would not choose.
  \item We extend the guaranteed bound to a quantity of interest via a duality
        argument (Section~\ref{sec:goal}), and we verify numerically
        that the resulting product bound is also insensitive to an isolated
        degenerate cell (Section~\ref{sec:numerics_goal}).
  \item We report numerical experiments (Section~\ref{sec:numerics}) on
        families of coarse, damaged, anisotropic and globally distorted meshes,
        and on the real meshes mentioned above. All the code is available and
        every table and figure in the paper is regenerated by it.
\end{enumerate}

\paragraph{What is, and what is not, new.}
None of the theory of Section~\ref{sec:eqflux} is new: the hypercircle identity
is due to Prager and Synge \cite{pragersynge1947}, the relaxation to cellwise
equilibration and the patchwise reconstruction are due to Vohral\'ik and
coworkers \cite{vohralik_notes,ernvohralik2015} and to Braess and Sch\"oberl
\cite{braess2008}. What we claim is a change of \emph{reading}. The absence of
any shape-regularity hypothesis in the upper bound is, in that literature, a
technical remark; here it is the object of the study, and we track it
deliberately --- isolating the one geometric ingredient that does appear
(Lemma~\ref{lem:poincare}), stating what breaks when it fails
(Remark~\ref{rem:nonconvex}), and measuring, rather than asserting, how much
sharpness distortion actually costs. The mapping between the quality measures
that meshing software reports and the constants that the analysis contains
(Section~\ref{sec:quality}) is, to our knowledge, not written down anywhere.
The experimental observations of Section~\ref{sec:numerics} --- that
the two reconstructions of the same guaranteed bound degrade under different
defects, that Bank--Weiser is sensitive to a collapsed edge and not to a
collapsed area, and that no single quality measure of
Definition~\ref{def:quality} predicts the local effectivity --- are, likewise,
new.

\paragraph{Scope.}
We restrict ourselves to the Poisson problem with a piecewise constant
coefficient, to conforming Lagrange elements, and --- in the experiments --- to
two space dimensions 
\franz{and low order piecewise linear finite element approximation\michel{.}}
\franz{These choices \michel{are} motivated by the following.}
The phenomenon we want to
isolate is geometric, and every additional feature of the model would make it
harder to attribute an effect to the mesh alone. The consequences of relaxing
each of these restrictions are discussed in Section~\ref{sec:conclusion}. We
also do not address the algebraic error: all linear systems below are solved by
a direct method.

\section{Notation}
\label{sec:setting}

Let $\Om \subset \Rbb^d$, $d \in \{2,3\}$, be a bounded Lipschitz polytopal
domain with boundary $\Gamma := \partial\Om$. For $\omega \subseteq \Om$ we
write $\ip{\cdot}{\cdot}_\omega$ for the $L^2(\omega)$ scalar product (as well as
for its vector-valued counterpart) and $\norm{\cdot}_\omega$ for the associated
norm; the subscript is dropped when $\omega = \Om$. We use the standard Sobolev
spaces $H^1(\Om)$, $V := H^1_0(\Om)$, and
\[
  \Hdiv := \{ \tauv \in [L^2(\Om)]^d \ : \ \dive \tauv \in L^2(\Om) \}.
\]
For $\tauv \in \Hdiv$ and $v \in V$, the Green formula
\begin{equation}\label{eq:green}
  \ip{\tauv}{\grad v}_\Om + \ip{\dive \tauv}{v}_\Om = 0
\end{equation}
holds (the boundary term vanishes since $v$ has zero trace).

\section{Model problem}
\label{sec:model}

Let $\mu > 0$ be a physical parameter, assumed piecewise constant on the mesh
introduced below, and let $f \in L^2(\Om)$. We consider the Poisson problem with
homogeneous Dirichlet boundary conditions:
\begin{equation}\label{eq:strong}
  - \franz{\mathrm{div}\ (\mu \grad u )}= f \quad \text{in } \Om, \qquad u = 0 \quad \text{on } \Gamma .
\end{equation}
Define the bilinear and linear forms
\begin{equation}\label{eq:forms}
  a(v,w) := \ip{\mu\grad v}{\grad w}_\Om, \qquad
  \ell(v) := \ip{f}{v}_\Om, \qquad v,w \in V,
\end{equation}
so that the weak formulation reads: find $u \in V$ such that
\begin{equation}\label{eq:weak}
  a(u,v) = \ell(v)\zuqi{,} \qquad \forall v \in V .
\end{equation}
By Lax--Milgram, \eqref{eq:weak} is well posed. We measure the error in the
\emph{energy norm}
\begin{equation}\label{eq:enorm}
  \enorm{v} := \norm{\mu^{1/2}\grad v}_\Om ,
  \qquad
  \enorm{v}_\omega := \norm{\mu^{1/2}\grad v}_\omega ,
\end{equation}
and we associate with it the dual norm on fluxes,
$\norm{\mu^{-1/2}\tauv}_\omega$. The \emph{exact flux} is
\begin{equation}\label{eq:exactflux}
  \sig := -\mu \grad u \in \Hdiv, \qquad \dive \sig = f \ \text{ in } \Om .
\end{equation}


\section{The mesh}
\label{sec:mesh}

Let $\Th$ be a partition of $\overline{\Om}$ into closed cells $K$ with
$\overline{\Om} = \bigcup_{K \in \Th} K$ and $\mathring{K} \cap \mathring{K'} =
\emptyset$ for $K \neq K'$. We assume only that $\Th$ is
\begin{assumption}[Admissible mesh]\label{ass:mesh}
\leavevmode
\begin{enumerate}[label=(M\arabic*),nosep]
  \item\label{it:M1} \emph{matching}: the intersection of two distinct cells is
        either empty, or a common vertex, edge or face (no hanging node)\footnote{\claudio{In the meshing community this is known as a manifold mesh.}}.
  \item\label{it:M2} each cell $K$ is the image of a fixed reference cell
        $\widehat{K}$ (simplex, quadrilateral/hexahedron, prism, \dots) by a
        smooth, invertible map $F_K : \widehat{K} \to K$ with
        $\det(\mathbb{J}_K) > 0$ everywhere on $\widehat{K}$, where
        $\mathbb{J}_K := \mathrm{D}F_K$;
  \item\label{it:M3} $\mu$ is constant on each $K$, with value $\mu_K$.
\end{enumerate}
\end{assumption}
We stress what is \emph{not} assumed: no shape regularity
($h_K / \rho_K \leq \gamma$, with $\rho_K$ the inradius), no quasi-uniformity
($\max_K h_K / \min_K h_K \leq \gamma$), no minimum nor maximum angle condition,
no bound on the number of cells sharing a vertex, no convexity of the cells. The
mesh may be a simplicial one, a tensor-product one, or a mixed one.

We write $h_K := \operatorname{diam}(K)$, $h := \max_{K\in\Th} h_K$, and
$\mathcal{F}_h$ for the set of faces, split into interior faces
$\mathcal{F}_h^{i}$ and boundary faces $\mathcal{F}_h^{b}$. The set of vertices
is $\mathcal{V}_h = \mathcal{V}_h^{i} \cup \mathcal{V}_h^{b}$ (interior /
boundary). For a vertex $\mathrm{a} \in \mathcal{V}_h$, $\omega_\mathrm{a}$
denotes the patch of cells sharing $\mathrm{a}$ and $\psi_\mathrm{a}$ the
associated $\michel{\mathbb{P}_1}$ hat function ($\operatorname{supp}\psi_\mathrm{a} =
\omega_\mathrm{a}$, $\sum_{\mathrm{a}} \psi_\mathrm{a} \equiv 1$). For a cell
$K$, $\omega_K := \bigcup \{ \omega_\mathrm{a} : \mathrm{a} \in K \}$.

\section{The finite element method}
\label{sec:fem}

Let $k \geq 1$ and let
\begin{equation}\label{eq:Vh}
  \Vh := \{ v_h \in V \ : \ v_h \circ F_K \in \mathbb{P}_k(\widehat{K})
        \ \ \forall K \in \Th \} \subset V
\end{equation}
be the conforming Lagrange finite element space ($\mathbb{Q}_k$ on
tensor-product cells); we refer to \cite{braess2007,ernguermond2021} for the
general theory. The discrete problem reads: find $u_h \in \Vh$ such that
\begin{equation}\label{eq:fem}
  a(u_h, v_h) = \ell(v_h)\zuqi{,} \qquad \forall v_h \in \Vh .
\end{equation}
Conformity $\Vh \subset V$ gives Galerkin orthogonality
\begin{equation}\label{eq:galerkin}
  a(u - u_h, v_h) = 0\zuqi{,} \qquad \forall v_h \in \Vh ,
\end{equation}
which we will use twice: to obtain the Prager--Synge identity, and --- crucially
--- to guarantee the solvability of the local Neumann problems of
Section~\ref{sec:patchwise}.

\begin{remark}[Conformity]
Because $\Vh \subset V$, no \emph{potential} reconstruction is needed and the
framework of \cite{vohralik_notes,ernvohralik2015} simplifies to a flux
reconstruction only. For nonconforming, DG or mixed discretizations, the
estimator gains a nonconformity term; see \cite{ernvohralik2015}.
\end{remark}

\section{Measuring the quality of a mesh}
\label{sec:quality}

This section collects the quality measures that a mesher produces, and confronts
them with the constants of the a posteriori analysis. Throughout,
$\mathbb{J}_K(\widehat\xv) = \mathrm{D}F_K(\widehat\xv) \in \Rbb^{d\times d}$
denotes the Jacobian matrix of the geometric map at
$\widehat\xv \in \widehat{K}$, with singular values
\[\varsigma_1(K) \geq \dots \geq \varsigma_d(K) > 0.\]

\subsection{Geometric and algebraic quality measures}

\begin{definition}[Usual quality measures]\label{def:quality}
For a cell $K$ we consider
\begin{enumerate}[label=(Q\arabic*),nosep]
  \item\label{it:Q1} the \emph{shape-regularity ratio}
        $\gamma_K := h_K / \rho_K$, $\rho_K$ being the diameter of the largest
        ball inscribed in $K$;
  \item\label{it:Q2} the \emph{aspect ratio}
        $\alpha_K := \varsigma_1(K) / \varsigma_d(K)$, i.e.\ the condition number
        of $\mathbb{J}_K$ in the spectral norm \michel{(for triangles we take
        $\widehat{K}$ equilateral here, so that $\alpha_K = 1$ characterises
        the equilateral cell and the value does not depend on the numbering
        of the vertices)};
  \item\label{it:Q3} the \emph{scaled Jacobian}
        \zuqi{
        $s_K := \min_{\widehat\xv}
        \Big[ \det(\mathbb{J}_K(\widehat\xv)) \big/ \prod_{i=1}^d
        \norm{(\mathbb{J}_K(\widehat\xv))_{:,i}}_2 \Big] \in [-1,1]$}
        equal to $1$ \michel{when the columns of $\mathbb{J}_K$ are
        orthogonal} and tending to $0$
        for a flat (sliver, cap, needle) element, and negative for an inverted
        one\michel{, i.e.\ one whose vertices are traversed in the wrong
        orientation (
        \zuqi{ 
        $\det(\mathbb{J}_K) < 0$}), which \ref{it:M2} excludes. For a
        triangle the bracket is the sine of the angle at the vertex onto which
        $F_K$ maps the origin of $\widehat{K}$; to make the measure
        independent of the numbering of the vertices we take, as mesh-quality
        software does, its minimum over the three corners,
        $s_K = \min_i \sin\theta_i(K)$, $\theta_i(K)$ being the angles of $K$.
        This equals $1/\sqrt2 \approx 0.71$ on the right triangles of a
        structured mesh (the value $s_{\min}$ reported for family \ref{it:F1}
        in Section~\ref{sec:numerics}), $\sqrt3/2$ on the equilateral
        triangle, and is of order $t$ for both the needle and the cap of
        family \ref{it:F3}}; 
  \item\label{it:Q4} the \emph{Frobenius condition number}
        $\kappa_K := \tfrac{1}{d}\,\norm{\mathbb{J}_K}_F
        \norm{\mathbb{J}_K^{-1}}_F \geq 1$ and the associated \emph{mean ratio}
        $1/\kappa_K$ \cite{knupp2001};
  \item\label{it:Q5} the extremal angles $\theta_{\min}(K)$,
        $\theta_{\max}(K)$ (dihedral angles in 3D).
\end{enumerate}
\end{definition}

These measures are not equivalent, and the differences matter. In 2D, a
triangle family with $\theta_{\max} \to \pi$ (\emph{cap}) and one with
$\theta_{\min}\to 0$ but $\theta_{\max}$ bounded away from $\pi$
(\emph{needle}) both have $\gamma_K \to \infty$, but only the first one destroys
the $H^1$ interpolation error \cite{babuskaaziz1976,shewchuk2002}. This is a
first hint that the classical shape-regularity ratio \ref{it:Q1} is a crude
proxy.

\subsection{Where does geometry enter the analysis?}
\label{sec:whereGeometry}

Let us anticipate on Section~\ref{sec:eqflux} and list the constants involved.

\begin{enumerate}[label=(C\arabic*),leftmargin=*]
\item \textbf{Poincar\'e constant on a cell.} The upper bound uses
      \[
        \norm{\varphi - \overline{\varphi}_K}_K \leq C_{P,K}\, h_K
        \norm{\grad \varphi}_K , \qquad \varphi \in H^1(K),
      \]
      where $\overline{\varphi}_K$ is the mean value of $\varphi$ over $K$. By
      Payne--Weinberger \cite{payneweinberger1960}, $C_{P,K} \leq 1/\pi$ for
      \emph{any convex} $K$, regardless of its shape --- a needle, a sliver or a
      cap triangle all admit $C_{P,K}=1/\pi$, and no better uniform value is
      available since $1/\pi$ is attained in the limit of a thin slab. This is \emph{the} reason why the
      guaranteed bound survives distortion. Convexity is however essential: for
      a non-convex cell the constant \michel{need not be} $1/\pi$ (see
      Remark~\ref{rem:nonconvex}).
\item \textbf{Cauchy--Schwarz.} Used several times, constant \zuqi{one}, no geometry.
\item \textbf{Local equilibration solvability.} Involves only Galerkin
      orthogonality \eqref{eq:galerkin} and the partition of unity, no geometry.
\item \textbf{Efficiency (lower bound).} The bound
      $\eta_K \lesssim \enorm{u-u_h}_{\omega_K} + \osc_K$ relies on
      bubble-function / discrete stability arguments whose constants \emph{do}
      degenerate with $\gamma_K$, with $\alpha_K$, and with the number of cells
      in a patch \cite{ernvohralik2015,braess2009}.
\item \textbf{Conditioning and a priori error.} Not the subject here, but recall
      that the condition number of the stiffness matrix and the interpolation
      constants degrade with $\alpha_K$ and $\theta_{\max}$
      \cite{shewchuk2002,duprez2019damage}.
\end{enumerate}

The message is summarised in Table~\ref{tab:sensitivity}.

\begin{table}[ht]
\centering
\small
\begin{tabular}{lccc}
\toprule
 & guaranteed & efficiency & cost \\
 & upper bound & (lower bound) & \\
\midrule
Explicit residual        & no (unknown $C$) & degenerates & cheap \\
Fully computable residual \cite{veeser2009,carstensen2014}
                         & yes, $C$ shape-dependent & degenerates & cheap \\
Zienkiewicz--Zhu \cite{zz1987}
                         & no & asympt.\ exact & very cheap \\
                         &    & if regular    &            \\
Bank--Weiser \cite{bankweiser1985,bulle2023}
                         & under saturation & degenerates & moderate \\
\textbf{Equilibrated flux}
                         & \textbf{yes, $C=1$,} & \textbf{degenerates} & \textbf{moderate} \\
                         & \textbf{geometry-free} & & \\
\bottomrule
\end{tabular}
\caption{Sensitivity of the main estimator families to the mesh quality. Only
the \emph{sharpness} of the equilibrated flux estimator is affected by
distortion; its \emph{reliability} is not.}
\label{tab:sensitivity}
\end{table}

\subsection{A quality field to be correlated with the estimator}

For the numerical section we define, for each cell, the vector of quality
indicators
\begin{equation}\label{eq:qualityvector}
  \mathbf{q}_K := \big( \gamma_K,\ \alpha_K,\ s_K,\ \kappa_K,\
                        \theta_{\min}(K),\ \theta_{\max}(K) \big),
\end{equation}
and we will study the empirical correlation between $\mathbf{q}_K$ and the local
effectivity index
\zuqi{
\begin{equation}\label{eq:localeff}
  I_K := \frac{\eta_K}{\enorm{u-u_h}_{\omega_K}} ,
  \qquad \eta_K :=  \eta_K^{\mathrm{F}}+\eta_K^{\mathrm{R}},
\end{equation}
the local indicators $\eta_K^{\mathrm{F}}$, $\eta_K^{\mathrm{R}}$ being defined
in \eqref{eq:etaF}--\eqref{eq:etaR} below. 
}

The claim to be tested is that the
global effectivity index \eqref{eq:Ieff} stays $O(1)$ under distortion for the
equilibrated flux estimator, whereas its residual and Zienkiewicz--Zhu
counterparts drift --- and, in the case of the \michel{latter}, \michel{drifts} \emph{below}
one, which is the dangerous direction.

\section{A posteriori error estimation}
\label{sec:eqflux}

We now come to the a posteriori error estimates. The guiding thread is the
equilibrated flux methodology: it is the only one, among those recalled in
Section~\ref{sec:comparison}, whose upper bound is both fully computable and
free of any assumption on the geometry of the cells.

\subsection{The Prager--Synge identity}
\label{sec:prager}

\begin{definition}[Admissible flux]\label{def:admissible}
A vector field $\pv$ is \emph{admissible} for \eqref{eq:strong} if
\begin{equation}\label{eq:AC}
  \pv \in \Hdiv \quad \text{and} \quad \dive \pv = f \ \text{ in } \Om .
  \tag{AC}
\end{equation}
\end{definition}

\begin{theorem}[Prager--Synge \cite{pragersynge1947}]\label{thm:prager}
Let $u \in V$ solve \eqref{eq:weak}, let $\pv$ satisfy \eqref{eq:AC} and let
$v \in V$ be arbitrary. Then
\begin{equation}\label{eq:prager}
  \norm{\mu^{-1/2}\left(\pv + \mu\grad v\right)}^2
  =
  \norm{\mu^{-1/2}\left(\pv - \sig\right)}^2 + \enorm{u-v}^2 .
\end{equation}
\end{theorem}

\begin{proof}
Write $\pv + \mu\grad v = (\pv - \sig) + \mu\grad(v-u)$, using
$\sig = -\mu\grad u$. Expanding the square,
\[
  \norm{\mu^{-1/2}(\pv + \mu\grad v)}^2
  = \norm{\mu^{-1/2}(\pv-\sig)}^2 + \enorm{u-v}^2
    + 2 \ip{\pv - \sig}{\grad(v-u)} .
\]
Set $w := v-u \in V$. By \eqref{eq:green} and \eqref{eq:AC},
$\ip{\pv}{\grad w} = -\ip{\dive\pv}{w} = -\ip{f}{w}$, while by \eqref{eq:weak},
$\ip{-\sig}{\grad w} = \michel{\ip{\mu\grad u}{\grad w}} = a(u,w) = \ip{f}{w}$. The two
contributions cancel, which is the \emph{orthogonality property} at the heart of
the identity.
\end{proof}

\begin{corollary}[Hypercircle bound]\label{cor:hypercircle}
Taking $v = u_h$ in \eqref{eq:prager} and dropping the first term on the right,
\begin{equation}\label{eq:APEB}
  \enorm{u-u_h} \ \leq \
  \inf_{\pv \text{ satisfying } \eqref{eq:AC}}
  \norm{\mu^{-1/2}\left(\pv + \mu\grad u_h\right)} .
\end{equation}
\end{corollary}

Three features of \eqref{eq:APEB} deserve emphasis, and they are the reason for
this work:
\begin{enumerate}[nosep]
  \item the multiplicative constant is \zuqi{one}: this is a true, quantitative bound
        on the error, not an asymptotic statement;
  \item \emph{every} admissible $\pv$ gives a valid bound, so the estimator is a
        \emph{certificate} that can be checked a posteriori;
  \item nothing in the proof involves the mesh --- in fact no mesh is needed at
        all, which suggests that the same machinery applies to meshless or
        immersed methods.
\end{enumerate}

The difficulty is of course that \eqref{eq:AC} is an equality of functions in
$L^2(\Om)$, and building a discrete field that satisfies it \emph{exactly} for a
general $f$ is out of reach in practice. The next subsection removes this
obstruction.

Before doing so, we record the convex-duality reading of \eqref{eq:prager},
which explains \emph{why} the constant is one.

\begin{proposition}[The estimator is a primal--dual gap]\label{prop:gap}
Define the primal and dual (complementary) energies
\begin{equation}\label{eq:energies}
  \mathcal{E}(v) := \tfrac12 a(v,v) - \ell(v), \quad v \in V,
  \qquad
  \mathcal{E}^*(\pv) := -\tfrac12\norm{\mu^{-1/2}\pv}^2, \quad
  \pv \text{ satisfying } \eqref{eq:AC} .
\end{equation}
Then $\min_{v\in V}\mathcal{E}(v) = \mathcal{E}(u)
= \mathcal{E}^*(\sig) = \max_{\pv}\mathcal{E}^*(\pv)$ --- there is no duality
gap --- and, for every $v\in V$ and every admissible $\pv$,
\begin{equation}\label{eq:gap}
  \mathcal{E}(v) - \mathcal{E}^*(\pv)
  = \tfrac12\enorm{u-v}^2 + \tfrac12\norm{\mu^{-1/2}(\pv-\sig)}^2
  = \tfrac12\norm{\mu^{-1/2}(\pv+\mu\grad v)}^2 .
\end{equation}
\end{proposition}

\begin{proof}
Since $a(u,v)=\ell(v)$, one has
$\mathcal{E}(v) = \tfrac12 a(v-u,v-u) - \tfrac12 a(u,u)
= \tfrac12\enorm{u-v}^2 - \tfrac12\norm{\mu^{1/2}\grad u}^2$. On the other hand,
for $\pv$ admissible, \eqref{eq:green} and \eqref{eq:weak} give
$\ip{\mu^{-1}\pv}{\sig} = -\ip{\pv}{\grad u} = \ip{\dive\pv}{u} = \ip{f}{u}
= \norm{\mu^{1/2}\grad u}^2$, whence
\[
  \tfrac12\norm{\mu^{-1/2}(\pv-\sig)}^2
  = \tfrac12\norm{\mu^{-1/2}\pv}^2 - \norm{\mu^{1/2}\grad u}^2
    + \tfrac12\norm{\mu^{-1/2}\sig}^2
  = -\mathcal{E}^*(\pv) - \tfrac12\norm{\mu^{1/2}\grad u}^2 ,
\]
using $\norm{\mu^{-1/2}\sig} = \norm{\mu^{1/2}\grad u}$. Adding the two
displays gives the first equality in \eqref{eq:gap}; the second one is
Theorem~\ref{thm:prager}. Taking $v=u$ and $\pv=\sig$ shows that the gap
vanishes, and non-negativity of \eqref{eq:gap} gives weak duality.
\end{proof}

\begin{remark}[Duality reading of the hypercircle]\label{rem:duality}
Proposition~\ref{prop:gap} says that the squared estimator
$\norm{\mu^{-1/2}(\sig_h+\mu\grad u_h)}^2$ is nothing but \emph{twice the
primal--dual gap} evaluated at the computed pair $(u_h,\sig_h)$: it is a
by-product of weak duality, which holds for any admissible pair and involves no
discretization whatsoever. This is why the constant is one, and this is the
viewpoint --- \emph{a posteriori estimation by duality} --- underlying the
functional error estimates of Repin \cite{repin2000,repin2008} and Han
\cite{han2005}; Bartels and Kaltenbach \cite{bartels2023} have recently
extended it to general convex minimisation, including nonsmooth and degenerate
problems, where the quadratic structure used above is unavailable.
\end{remark}

\subsection{The abstract framework of Vohral\'ik: relaxing the admissibility
            condition}
\label{sec:vohralik}

Following \cite[Chap.~7]{vohralik_notes} and \cite{ernvohralik2015} we only
require equilibration \emph{in the mean on each cell}.

\begin{definition}[Equilibrated flux reconstruction]\label{def:eqflux}
A field $\sig_h$ is an \emph{equilibrated flux reconstruction} if
\begin{equation}\label{eq:EF}
  \sig_h \in \Hdiv
  \quad \text{and} \quad
  \ip{\dive \sig_h - f}{1}_K = 0 \qquad \forall K \in \Th .
  \tag{EF}
\end{equation}
\end{definition}

\begin{lemma}[Poincar\'e on a cell]\label{lem:poincare}
Let $K$ be a \emph{convex} cell and $\varphi \in H^1(K)$ with mean value
$\overline{\varphi}_K$. Then
\begin{equation}\label{eq:poincare}
  \norm{\varphi - \overline{\varphi}_K}_K
  \leq \frac{h_K}{\pi} \norm{\grad\varphi}_K .
\end{equation}
\end{lemma}

\begin{proof}
This is the Payne--Weinberger inequality \cite{payneweinberger1960}; the
constant $h_K/\pi$ is optimal (attained in the limit of a thin
rectangle/cylinder) and depends on $K$ only through its diameter.
\end{proof}

\begin{remark}[Non-convex cells]\label{rem:nonconvex}
Inequality \eqref{eq:poincare} may fail on a non-convex cell. Three remedies keep
the bound guaranteed: (i) split each non-convex cell into convex sub-cells for
the purpose of evaluating $\eta_K^{\mathrm{R}}$ only, and use $h_{K'}/\pi$ on
each sub-cell $K'$ --- this only \emph{improves} the estimate since
$h_{K'}\le h_K$; (ii) use the constant $C_{P,K}h_K$ with $C_{P,K}$ computed from
a local eigenvalue problem (a small, purely local computation); (iii) \michel{when
$f$ is a piecewise polynomial of degree at most $p$,} make the
term vanish by using a reconstruction with $\dive\sig_h = \Pi_p f$\michel{, see
Proposition~\ref{prop:oscillation}; for a general $f$ the data oscillation
\eqref{eq:osc} remains and still involves the Poincar\'e constant of the
cell}. \michel{In our numerical experiments all cells are triangles, hence
convex, and Lemma~\ref{lem:poincare} applies directly; the reconstructions
used there satisfy $\dive\sig_h = \Pi_p f$ exactly, so that
$\eta_K^{\mathrm{R}}$ reduces to the data oscillation \eqref{eq:osc}.}
\end{remark}

We now define the two families of local indicators, for each $K\in\Th$:
\zuqi{
\begin{align}
  \eta_K^{\mathrm{R}}
    &:= \frac{h_K}{\pi}\, \mu_K^{-1/2}
        \norm{f - \dive \sig_h}_K
    && \text{(residual, or equilibration defect, indicator)} \label{eq:etaR}\\
 \eta_K^{\mathrm{F}}
    &:= \norm{\mu^{-1/2}\left(\sig_h + \mu \grad u_h\right)}_K
    && \text{(flux, or constitutive relation, indicator)}.
    \label{eq:etaF}
\end{align}
}
Both are computable once $u_h$ and $\sig_h$ are at hand: they involve no unknown
constant and no reference to the exact solution.

\begin{theorem}[Guaranteed upper bound]\label{thm:upper}
Let $u$ solve \eqref{eq:weak}, $u_h$ solve \eqref{eq:fem}, and let $\sig_h$
satisfy \eqref{eq:EF}. Assume the cells are convex (see
Remark~\ref{rem:nonconvex} otherwise). Then
\begin{equation}\label{eq:upper}
  \enorm{u-u_h}^2
  \ \leq \
  \sum_{K\in\Th}
  \left( \eta_K^{\mathrm{R}} + \eta_K^{\mathrm{F}} \right)^2
  =: \eta^2 .
\end{equation}
\end{theorem}

\begin{proof}
Since $u - u_h \in V$, and by definition of the energy norm,
\begin{equation}\label{eq:dualchar}
  \enorm{u-u_h}
  = \sup_{\varphi \in V,\ \enorm{\varphi}=1} a(u-u_h,\varphi) .
\end{equation}
Fix such a $\varphi$. Using \eqref{eq:weak},
\[
  a(u-u_h,\varphi) = \ip{f}{\varphi} - \ip{\mu\grad u_h}{\grad\varphi}.
\]
Insert $\pm \dive\sig_h$ and use the Green formula \eqref{eq:green} for
$\sig_h \in \Hdiv$ and $\varphi \in V$, i.e.\
$\ip{\dive\sig_h}{\varphi} = -\ip{\sig_h}{\grad\varphi}$:
\begin{equation}\label{eq:splitting}
  a(u-u_h,\varphi)
  = \underbrace{\ip{f - \dive\sig_h}{\varphi}}_{=: T_{\mathrm{R}}}
    \ \underbrace{-\ip{\sig_h + \mu\grad u_h}{\grad\varphi}}_{=: T_{\mathrm{F}}} .
\end{equation}
For $T_{\mathrm{F}}$, a cellwise Cauchy--Schwarz inequality gives
\[
  T_{\mathrm{F}}
  \leq \sum_{K\in\Th}
       \norm{\mu^{-1/2}(\sig_h + \mu\grad u_h)}_K \,
       \norm{\mu^{1/2}\grad\varphi}_K
  = \sum_{K\in\Th} \eta_K^{\mathrm{F}}\, \enorm{\varphi}_K .
\]
For $T_{\mathrm{R}}$, split over the cells and use \eqref{eq:EF} to subtract the
cell mean of $\varphi$ for free:
\[
  T_{\mathrm{R}}
  = \sum_{K\in\Th} \ip{f - \dive\sig_h}{\varphi}_K
  = \sum_{K\in\Th} \ip{f - \dive\sig_h}{\varphi - \overline{\varphi}_K}_K .
\]
Cauchy--Schwarz and Lemma~\ref{lem:poincare} then yield
\[
  \begin{aligned}
  T_{\mathrm{R}}
  &\leq \sum_{K\in\Th}
       \norm{f-\dive\sig_h}_K \, \frac{h_K}{\pi}\norm{\grad\varphi}_K \\
  &= \sum_{K\in\Th}
       \Big[ \tfrac{h_K}{\pi}\mu_K^{-1/2}\norm{f-\dive\sig_h}_K \Big]
       \Big[ \mu_K^{1/2}\norm{\grad\varphi}_K \Big]
  = \sum_{K\in\Th} \eta_K^{\mathrm{R}}\, \enorm{\varphi}_K ,
  \end{aligned}
\]
where we inserted $1 = \mu_K^{-1/2}\mu_K^{1/2}$, licit because $\mu$ is constant
on $K$ by \ref{it:M3}. Summing the two bounds and applying a last, discrete
Cauchy--Schwarz inequality,
\[
  a(u-u_h,\varphi)
  \leq \sum_{K\in\Th}(\eta_K^{\mathrm{R}}+\eta_K^{\mathrm{F}})\enorm{\varphi}_K
  \leq \Big( \sum_{K\in\Th}(\eta_K^{\mathrm{R}}+\eta_K^{\mathrm{F}})^2\Big)^{1/2}
       \Big( \sum_{K\in\Th}\enorm{\varphi}_K^2\Big)^{1/2} ,
\]
and the last factor equals $\enorm{\varphi}=1$. Taking the supremum over
$\varphi$ in \eqref{eq:dualchar} concludes.
\end{proof}

\begin{remark}[What was \emph{not} used]
The proof used: the conformity $\Vh\subset V$ (through $u-u_h \in V$), the
$H(\dive)$-conformity of $\sig_h$ (through the Green formula), the cellwise
equilibration \eqref{eq:EF}, and the Poincar\'e inequality on convex cells. No
shape regularity, no minimum/maximum angle, no quasi-uniformity, no bound on the
patch cardinality, no asymptotic regime in $h$. In particular
\eqref{eq:upper} is valid on the coarsest and most distorted mesh, and even for
$k=1$ on a single element.
\end{remark}

\begin{remark}[Non-homogeneous Dirichlet data]\label{rem:nonhomog}
Suppose $u = g$ on $\Gamma$ with $g$ in the trace space of $H^1(\Om)$, and let
$u_h \in \Vh$ satisfy \eqref{eq:fem} for all $v_h \in \Vh \cap H^1_0(\Om)$ with
its boundary values fixed to some $g_h$. Theorem~\ref{thm:upper} goes through
verbatim \emph{provided} $u - u_h \in H^1_0(\Om)$, i.e.\ provided
$g_h = g$ exactly on $\Gamma$ --- the only place where the boundary condition
entered the proof was through $u - u_h \in V$ and the vanishing of the boundary
term in \eqref{eq:green}. If, as is usual, $g_h$ is merely an interpolant of
$g$, then $u - u_h \notin H^1_0(\Om)$ and an additional term appears, measuring
the distance from $g_h$ to $g$ in $H^{1/2}(\Gamma)$; bounding it computably
requires an argument of a different nature, and the constant it brings does
depend on the shape of the boundary cells. This is why the main experiments of
Section~\ref{sec:numerics} are built on test cases with homogeneous Dirichlet
data. The L-shaped test case \ref{it:T2}, whose exact solution
$r^{2/3}\sin(2\theta/3)$ has non-polynomial boundary values, therefore falls
outside the theorem; we report it separately in
Section~\ref{sec:numerics_lshape} as an experiment on what happens when a
hypothesis is dropped, not as evidence for the bound. This is a genuine gap in
the theory presented here, not an oversight; see \cite{ernvohralik2015} for the
treatment of general boundary data.
\end{remark}

\begin{proposition}[Data oscillation form]\label{prop:oscillation}
Assume the reconstruction satisfies the stronger property
$\dive\sig_h = \Pi_p f$, where $\Pi_p$ denotes the $L^2$-orthogonal projection
onto the space of piecewise polynomials of degree at most $p$ (broken, i.e.\
cellwise). Then $\eta_K^{\mathrm{R}}$ reduces to the \emph{data oscillation}
\begin{equation}\label{eq:osc}
  \osc_K(f) := \frac{h_K}{\pi}\,\mu_K^{-1/2}\norm{f - \Pi_p f}_K ,
\end{equation}
which is exactly computable, involves neither $u_h$ nor $\sig_h$, and is of
order $h_K^{p+2}$ on each cell where $f \in H^{p+1}(K)$ --- \michel{at least}
one order better than $\eta_K^{\mathrm{F}}$\michel{, which is of order $h^{k}$,
as soon as $p \geq k-1$}. It vanishes identically when $f$ is itself a
piecewise polynomial of degree $\le p$.
\end{proposition}

\begin{remark}[When may the residual term be dropped?]\label{rem:dropR}
Only under the hypothesis of Proposition~\ref{prop:oscillation}, and even then
it should be \emph{computed} rather than dropped, since it is cheap. For a
reconstruction that merely satisfies \eqref{eq:EF}, $\eta_K^{\mathrm{R}}$ is of
the same order as $\eta_K^{\mathrm{F}}$ and neglecting it destroys the
guarantee. Classical presentations are often silent on this distinction; the
FEniCS script of Section~\ref{sec:code} evaluates both terms.
\end{remark}

\subsection{Efficiency, and where the mesh quality strikes back}
\label{sec:efficiency}

Reliability is the good news; efficiency is where geometry returns. We state the
known result informally, as we shall not use it quantitatively.

\begin{theorem}[Local efficiency \michel{\cite{braess2009,ernvohralik2015,ernvohralik2020,caizhang2012}}]
\label{thm:efficiency}
For the patchwise reconstruction of Section~\ref{sec:patchwise} on a simplicial
mesh, there holds
\begin{equation}\label{eq:efficiency}
  \eta_K^{\mathrm{F}}
  \ \leq \ C_{\mathrm{st}}\big(\gamma_{\omega_K}\big)\,
  \Big( \enorm{u-u_h}_{\omega_K} + \osc_{\omega_K}(f) \Big),
  \qquad
  \osc_{\omega_K}(f)^2 := \sum_{K'\subset\omega_K}\osc_{K'}(f)^2 ,
\end{equation}
where $C_{\mathrm{st}}$ depends on the shape-regularity
$\gamma_{\omega_K} := \max_{K'\subset\omega_K}\gamma_{K'}$ of the patch and on
its cardinality, but neither on the polynomial degree $k$\michel{, which is the
content of \cite{braess2009,ernvohralik2020},} nor
on the diffusion contrast\michel{, which for conforming elements is the result of
Cai and Zhang \cite{caizhang2012}, obtained without any quasi-monotonicity
assumption on $\mu$}.
\end{theorem}

Combining Theorems~\ref{thm:upper} and \ref{thm:efficiency}, the effectivity
index
\begin{equation}\label{eq:Ieff}
  I := \frac{\eta}{\enorm{u-u_h}}
\end{equation}
satisfies $1 \leq I \leq C_{\mathrm{st}}(\gamma)$ up to oscillation: the
\emph{lower} bound $I\geq 1$ is unconditional, whereas the \emph{upper} bound
degrades with the mesh quality. Distortion can therefore make the estimator
pessimistic, never unreliable. From an engineering standpoint this is exactly
the right failure mode --- the user is never told that the solution is more
accurate than it is. Quantifying the growth of $I$ with the quality vector
\eqref{eq:qualityvector} is one of the objectives of Section~\ref{sec:numerics}.

\begin{remark}[Locally damaged meshes]
Following \cite{duprez2019damage}, if only a small subset
$\Th^{\mathrm{dam}} \subset \Th$ of cells is distorted, one expects
$C_{\mathrm{st}}$ to degrade only on the patches touching
$\Th^{\mathrm{dam}}$, so that the global effectivity index remains $O(1)$: the
few bad patches contribute a bounded amount to the sum. The relevant parameter
should then be the \emph{diameter of the damaged patch}, not the quality of an
isolated cell. This is a conjecture we probe numerically in
Section~\ref{sec:numerics_damaged}.
\end{remark}

\subsection{A first (expensive) estimator: global mixed
            \texorpdfstring{$H(\dive)$}{H(div)} formulation}
\label{sec:mixed}

We now describe two ways of producing a $\sig_h$ satisfying \eqref{eq:EF}. The
first, presented here, is trivial to implement and expensive; the second,
Section~\ref{sec:patchwise}, is the method of choice.
Fix $p \geq 0$ and let
\begin{equation}\label{eq:mixedspaces}
  Q_h := \{ q_h \in L^2(\Om) : q_h|_K \in \mathbb{P}_p(K)
            \ \ \forall K\in\Th \} ,
  \ 
  \Vfl_h \subset \Hdiv \ \text{ with } \ \dive\Vfl_h = Q_h .
\end{equation}
Admissible choices on a simplicial mesh are the Raviart--Thomas pair
$\Vfl_h = \RT_p(\Th)$ \cite{raviartthomas1977} and the Brezzi--Douglas--Marini
pair $\Vfl_h = \BDM_{p+1}(\Th)$ \cite{bdm1985}. The degree shift matters:
$\dive\RT_q = \mathbb{P}_q$ because $\RT_q = [\mathbb{P}_q]^d \oplus
\xv\,\widetilde{\mathbb{P}}_q$, whereas $\BDM_q = [\mathbb{P}_q]^d$ only gives
$\dive\BDM_q = \mathbb{P}_{q-1}$; pairing $\BDM_p$ with $\mathbb{P}_p$ would be
unstable. Both pairs above are inf-sup stable
\cite{boffibrezzifortin2013}. We then solve the dual mixed formulation of
\eqref{eq:strong}: find
$(\sig_h, \lambda_h) \in \Vfl_h \times Q_h$ such that
\begin{equation}\label{eq:mixedpb}
\left\{
\begin{aligned}
  \ip{\mu^{-1}\sig_h}{\tauv_h}_\Om - \ip{\lambda_h}{\dive\tauv_h}_\Om
    &= 0 && \forall \tauv_h \in \Vfl_h, \\
  \ip{\dive\sig_h}{q_h}_\Om &= \ip{f}{q_h}_\Om && \forall q_h \in Q_h .
\end{aligned}
\right.
\end{equation}
Since $\dive\Vfl_h = Q_h$, the second equation is equivalent to
$\dive\sig_h = \Pi_p f$; hence \eqref{eq:EF} holds \emph{exactly}, up to
round-off, and Proposition~\ref{prop:oscillation} applies. The first equation is
the weak form of the constitutive relation $\mu^{-1}\sig + \grad\lambda = 0$
--- the Dirichlet condition $u=0$ being natural here, it is enforced weakly and
no essential constraint is imposed on $\Vfl_h$. Thus $\lambda_h$ approximates
$u$ and $\sig_h$ approximates the exact flux $\sig = -\mu\grad u$ directly in
$\Hdiv$; see \cite{carstensen2013} for the analysis of the resulting estimator.

\begin{remark}[Pros and cons]
Pros: three lines of \texttt{FEniCS}/\texttt{FreeFEM}, no local solver to write,
and it produces an \emph{exactly} equilibrated flux. Cons: it requires the
solution of a \emph{global} saddle-point problem, typically more expensive than
the primal problem it is supposed to certify, which defeats the purpose in a
production run. It is however perfect as a reference implementation, and this is
how we use it in Section~\ref{sec:numerics}.
\end{remark}

\begin{remark}[Do not simply project the discrete flux]\label{rem:noproject}
A tempting shortcut is to take $\sig_h := \Pi_{\Vfl_h}(-\mu\grad u_h)$, the
$L^2$-projection of the discrete flux onto $\RT_p$. The result does belong to
$\Hdiv$, so the Green formula used in Theorem~\ref{thm:upper} is still valid,
and the bound \eqref{eq:upper} still \emph{holds}. But nothing forces
\eqref{eq:EF}, so $\eta_K^{\mathrm{R}}$ is no longer a data oscillation: it is a
genuine term of the same order as $\eta_K^{\mathrm{F}}$, and the bound, though
valid, becomes useless in practice. Such a projection is essentially a
Zienkiewicz--Zhu recovery in disguise (Section~\ref{sec:comparison}); the whole
point of an \emph{equilibrated} reconstruction is precisely to make the residual
term collapse to \eqref{eq:osc}.
\end{remark}

\subsection{Patchwise reconstruction (Braess--Sch\"oberl / Ern--Vohral\'ik)}
\label{sec:patchwise}

The efficient construction localises \eqref{eq:mixedpb} on vertex patches, using
the partition of unity $\sum_{\mathrm{a}\in\mathcal{V}_h}\psi_\mathrm{a}=1$. The
idea is to build separately, on each patch, an approximation of the
\emph{localised exact flux} $\psi_\mathrm{a}\sig$, and to sum them up. The
divergence of that target is, by the Leibniz rule and \eqref{eq:exactflux},
\begin{equation}\label{eq:divlocal}
  \dive\left(\psi_\mathrm{a}\sig\right)
  = \psi_\mathrm{a}\dive\sig + \grad\psi_\mathrm{a}\cdot\sig
  = \psi_\mathrm{a} f - \mu\,\grad u\cdot\grad\psi_\mathrm{a} ,
\end{equation}
which dictates the local right-hand side below --- note the \emph{minus} sign,
which is what makes the local Neumann problem compatible.

For $\mathrm{a}\in\mathcal{V}_h$, set
\begin{equation}\label{eq:localspaces}
\begin{aligned}
  \Vfl_h^{\mathrm{a}} &:= \{ \tauv_h \in \RT_p(\omega_\mathrm{a}) :
      \tauv_h\cdot\nv = 0 \text{ on } \partial\omega_\mathrm{a}
      \setminus \Gamma \}, \\
  Q_h^{\mathrm{a}} &:= \{ q_h \in L^2(\omega_\mathrm{a}) :
      q_h|_K \in \mathbb{P}_p(K) \ \forall K \subset \omega_\mathrm{a},
      \text{ and } \ip{q_h}{1}_{\omega_\mathrm{a}} = 0 \ \text{ if }
      \mathrm{a}\in\mathcal{V}_h^{i}\},
\end{aligned}
\end{equation}
so that $\dive\Vfl_h^{\mathrm{a}} = Q_h^{\mathrm{a}}$. The zero-mean constraint
on $Q_h^{\mathrm{a}}$, and the homogeneous normal trace on the whole of
$\partial\omega_\mathrm{a}$ that goes with it, reflect the pure Neumann nature
of the local problem at an interior vertex; at a boundary vertex the normal
trace is left free on $\partial\omega_\mathrm{a}\cap\Gamma$ and no mean-value
constraint is needed. Solve, for each $\mathrm{a}$, the local mixed problem:
find $(\sig_h^{\mathrm{a}}, r_h^{\mathrm{a}}) \in
\Vfl_h^{\mathrm{a}}\times Q_h^{\mathrm{a}}$ such that
\begin{equation}\label{eq:localpb}
\left\{
\begin{aligned}
  \ip{\mu^{-1}\sig_h^{\mathrm{a}}}{\tauv_h}_{\omega_\mathrm{a}}
  - \ip{r_h^{\mathrm{a}}}{\dive\tauv_h}_{\omega_\mathrm{a}}
    &= -\ip{\psi_\mathrm{a}\grad u_h}{\tauv_h}_{\omega_\mathrm{a}}
    && \forall \tauv_h \in \Vfl_h^{\mathrm{a}}, \\[2pt]
  \ip{\dive\sig_h^{\mathrm{a}}}{q_h}_{\omega_\mathrm{a}}
    &= \ip{g_\mathrm{a}}{q_h}_{\omega_\mathrm{a}}
    && \forall q_h \in Q_h^{\mathrm{a}},
\end{aligned}
\right.
\end{equation}
with the local right-hand side suggested by \eqref{eq:divlocal},
\begin{equation}\label{eq:ga}
  g_\mathrm{a} := \psi_\mathrm{a} f
                 - \mu\,\grad u_h \cdot \grad\psi_\mathrm{a} .
\end{equation}
Note that the factor $\mu^{-1}$ in the first equation of \eqref{eq:localpb}
cancels the factor $\mu$ carried by the target $-\psi_\mathrm{a}\mu\grad u_h$,
which is why $\mu$ does not appear on its right-hand side. Equivalently, and
this is the form to keep in mind, $\sig_h^{\mathrm{a}}$ is the constrained
minimiser
\begin{equation}\label{eq:localmin}
  \sig_h^{\mathrm{a}}
  = \arg\min_{\substack{\tauv_h \in \Vfl_h^{\mathrm{a}} \\
      \dive\tauv_h = \Pi_p g_\mathrm{a}}}
    \norm{\mu^{-1/2}\left(\tauv_h + \psi_\mathrm{a}\,\mu\grad u_h\right)}
      _{\omega_\mathrm{a}} ,
\end{equation}
i.e.\ the equilibrated field closest, in the $\mu^{-1}$-weighted $L^2$ norm, to
the localised discrete flux $-\psi_\mathrm{a}\mu\grad u_h$; here $r_h^{\mathrm{a}}$
is the Lagrange multiplier of the divergence constraint. Finally set
\begin{equation}\label{eq:assembly}
  \sig_h := \sum_{\mathrm{a}\in\mathcal{V}_h} \sig_h^{\mathrm{a}} ,
\end{equation}
each $\sig_h^{\mathrm{a}}$ being extended by zero outside $\omega_\mathrm{a}$.

\begin{proposition}[The construction is well posed and equilibrated]
\label{prop:patchwise}
Let $k\geq 1$, so that $\psi_\mathrm{a}\in\Vh$ (see \cite{ernvohralik2015}
for the choice of $p$ as a function of $k$). Then:
\begin{enumerate}[label=(\roman*),nosep]
  \item For $\mathrm{a}\in\mathcal{V}_h^{i}$, problem \eqref{eq:localpb} is well
        posed. Indeed, the pure Neumann compatibility condition
        $\ip{g_\mathrm{a}}{1}_{\omega_\mathrm{a}} = 0$ required for
        $\Pi_p g_\mathrm{a}$ to belong to $\dive\Vfl_h^{\mathrm{a}}
        = Q_h^{\mathrm{a}}$ reads
        \[
          \ip{g_\mathrm{a}}{1}_{\omega_\mathrm{a}}
          = \ip{f}{\psi_\mathrm{a}}_\Om
            - \ip{\mu\grad u_h}{\grad \psi_\mathrm{a}}_\Om
          = \ell(\psi_\mathrm{a}) - a(u_h,\psi_\mathrm{a}) = 0 ,
        \]
        which is exactly the discrete problem \eqref{eq:fem} tested with the hat
        function $\psi_\mathrm{a}\in\Vh$. This is where the \emph{minus} sign in
        \eqref{eq:ga} is essential: with a plus sign one would obtain
        $2\,\ell(\psi_\mathrm{a})\neq 0$ and the local problem would have no
        solution.
  \item Each $\sig_h^{\mathrm{a}}$ has zero normal trace on
        $\partial\omega_\mathrm{a}\setminus\Gamma$, so its extension by zero has
        no normal jump across $\partial\omega_\mathrm{a}$ and belongs to
        $\Hdiv$; hence so does the finite sum $\sig_h$.
  \item $\dive\sig_h^{\mathrm{a}} = \Pi_p g_\mathrm{a}$ on $\omega_\mathrm{a}$.
        For a boundary vertex this is immediate since $Q_h^{\mathrm{a}}$ is the
        full broken polynomial space. For an interior vertex, the second
        equation of \eqref{eq:localpb} only gives that
        $\dive\sig_h^{\mathrm{a}} - \Pi_p g_\mathrm{a}$ is $L^2$-orthogonal to
        all zero-mean $q_h$, i.e.\ is a constant $c$ on $\omega_\mathrm{a}$; but
        $\ip{\dive\sig_h^{\mathrm{a}}}{1}_{\omega_\mathrm{a}}
        = \int_{\partial\omega_\mathrm{a}}\sig_h^{\mathrm{a}}\cdot\nv = 0
        = \ip{\Pi_p g_\mathrm{a}}{1}_{\omega_\mathrm{a}}$ by (i), so $c=0$.
  \item Consequently, using $\sum_\mathrm{a}\psi_\mathrm{a}\equiv 1$ and hence
        $\sum_\mathrm{a}\grad\psi_\mathrm{a}\equiv \boldsymbol{0}$,
        \[
          \dive\sig_h = \sum_\mathrm{a} \Pi_p g_\mathrm{a}
          = \Pi_p\Big( f\sum_\mathrm{a}\psi_\mathrm{a}
            - \mu\grad u_h\cdot\sum_\mathrm{a}\grad\psi_\mathrm{a}\Big)
          = \Pi_p f .
        \]
        Therefore \eqref{eq:EF} holds and Proposition~\ref{prop:oscillation}
        applies: the residual indicator reduces to the data oscillation
        \eqref{eq:osc}.
\end{enumerate}
\end{proposition}

The whole procedure is summarised as follows.

\begin{algorithm}[Guaranteed error bound, patchwise version]\label{alg:main}
\leavevmode
\begin{enumerate}[label=\arabic*.,nosep,leftmargin=*]
  \item Solve the primal problem \eqref{eq:fem} for $u_h$.
  \item For each vertex $\mathrm{a}\in\mathcal{V}_h$ (in parallel): assemble
        $g_\mathrm{a}$ by \eqref{eq:ga} and solve the local mixed problem
        \eqref{eq:localpb} for $\sig_h^{\mathrm{a}}$.
  \item Assemble $\sig_h := \sum_\mathrm{a}\sig_h^{\mathrm{a}}$ by
        \eqref{eq:assembly}.
  \item \emph{Certificate check} (cheap, optional but recommended): verify that
        $\max_K |\ip{\dive\sig_h - f}{1}_K|$ is at round-off level.
  \item For each cell $K$, evaluate $\eta_K^{\mathrm{F}}$ by \eqref{eq:etaF} and
        $\eta_K^{\mathrm{R}} = \osc_K(f)$ by \eqref{eq:osc}.
  \item Return $\eta = \big(\sum_K(\eta_K^{\mathrm{R}}+\eta_K^{\mathrm{F}})^2
        \big)^{1/2}$. By Theorem~\ref{thm:upper},
        $\enorm{u-u_h}\leq\eta$, whatever the mesh.
\end{enumerate}
\end{algorithm}

\begin{remark}[Cost and parallelism]
Step 2 requires $\#\mathcal{V}_h$ independent mixed solves, each of a size
depending only on $p$ and on the number of cells in the patch, and each
independent of the others: the step is embarrassingly parallel and its
complexity is linear in $\#\mathcal{V}_h$, against a global saddle-point solve
for the reconstruction of Section~\ref{sec:mixed}. On the meshes of
Section~\ref{sec:numerics} with $k=1$ and $p=1$, each local system has at most
$43$ unknowns, so the arithmetic cost of step 2 is a few hundred flops per
vertex. We report that algebraic size rather than a wall-clock comparison: our
driver loops over the patches in \texttt{Python}, and its running time measures
the interpreter, not the algorithm (Section~\ref{sec:numerics_pw}).
\end{remark}

\begin{remark}[The correction term of the classical constructions]
\label{rem:ziguigui}
An alternative and historically earlier route (Destuynder--M\'etivet
\cite{destuynder1999}, Lad\`eveze \cite{ladeveze2004}) is to build a first field
$\widetilde{\pv}_h := \sum_\mathrm{a}\pv_h^{\mathrm{a}}$ that equilibrates only
approximately, and then to add a cellwise correction $\delta\pv_K$, obtained by
solving a local Neumann problem on $K$, so that
$\pv_h|_K := \widetilde{\pv}_h|_K + \delta\pv_K$ is exactly admissible. The
subtlety, easily overlooked, is that \emph{the correction must be computed
exactly} for the bound to remain guaranteed; if it is computed in a finite
dimensional space, it is not exactly admissible, and one is back to the relaxed
setting of Definition~\ref{def:eqflux}. Presentations that assert that the
correction is ``negligible'' implicitly assume a fine and shape-regular mesh
--- which is precisely the assumption we do not make. The formulation \eqref{eq:localpb}
sidesteps the issue altogether: the constraint $\dive\tauv_h = \Pi_p
g_\mathrm{a}$ is imposed \emph{inside} the discrete space, and the residual
defect is exactly the computable oscillation of
Proposition~\ref{prop:oscillation}.
\end{remark}

\begin{remark}[Distorted cells: implementation pitfalls]
On strongly distorted or nearly degenerate cells the local matrices in
\eqref{eq:localpb} become badly conditioned. Practical remedies: (i) scale the
$\RT$ basis by the face measures; (ii) solve \eqref{eq:localmin} in the
constrained least-squares form with an orthonormalised basis of the divergence
constraint; (iii) use a direct solver with pivoting on each patch --- the
matrices are tiny. Note that a badly conditioned \emph{local} problem does not
compromise the guarantee: any computed $\sig_h$ that lies in $\Hdiv$ and passes
the check $\ip{\dive\sig_h - f}{1}_K = 0$ (verifiable a posteriori, to round-off)
gives a valid bound, however suboptimal. Note also that the approximation
properties of the $\RT_p$ interpolant on strongly anisotropic tetrahedra are
known and better than the isotropic theory suggests \cite{acosta2011}, which is
encouraging for the \michel{anisotropic} mesh \michel{family} \ref{it:F5} below.
\end{remark}

\section{Comparison with other estimators}
\label{sec:comparison}

\subsection{Explicit residual}

The classical explicit residual estimator reads, on a simplicial mesh,
\begin{equation}\label{eq:residual}
  \eta_K^{\mathrm{res}} :=
  \Big( C_1^2\, h_K^2\, \mu_K^{-1}\norm{f + \zuqi{\mathrm{div}\ (\mu \grad u\michel{_h} )}}_K^2
      + C_2^2 \sum_{F \subset \partial K \cap \mathcal{F}_h^i}
        h_F\, \mu_F^{-1}\norm{[\![ \mu\grad u_h\cdot\nv ]\!]}_F^2
  \Big)^{1/2},
\end{equation}
where \zuqi{$[\![\cdot]\!]_F$} is the jump across the face $F$, $h_F:=\operatorname{diam}
F$, $\mu_F$ the larger of the two adjacent values of $\mu$, and $C_1,C_2$ are
fixed normalisation constants. One has $\enorm{u-u_h} \leq C \big(\sum_K
(\eta_K^{\mathrm{res}})^2 \big)^{1/2}$ with $C$ depending on the
shape-regularity of $\Th$ through the
constants of the Cl\'ement \cite{clement1975} / Scott--Zhang
\cite{scottzhang1990} quasi-interpolation; see
\cite{bernardi2001,verfurth2013} for a careful discussion. Two defects, both
fatal in our context: $C$ is not computable, and $C \to \infty$ as the elements
degenerate (for a triangle, $C$ grows like $(\sin\theta_{\max})^{-1}$ or worse).
Fully computable variants \cite{veeser2009,carstensen2014} replace $C$ by an
explicit expression, which however still involves the element shapes. We
nevertheless include $\eta^{\mathrm{res}}$ in the comparison, with the usual
engineering choice $C=1$, precisely to exhibit its unreliability on distorted
meshes.

\subsection{\michel{Recovery estimators of Zienkiewicz--Zhu type}}

With $G_h : [L^2]^d \to [\Vh]^d$ a gradient recovery operator (patch averaging,
superconvergent patch recovery, \dots),
\begin{equation}\label{eq:zz}
  \eta_K^{\mathrm{ZZ}} := \norm{\mu^{1/2}\left(G_h(\grad u_h)
      - \grad u_h\right)}_K .
\end{equation}
Cheap, popular, asymptotically exact on structured shape-regular meshes thanks to
superconvergence \cite{zz1992a,zz1992b} --- and with no guarantee whatsoever
otherwise. \michel{The variant implemented below takes for $G_h$ the global
$L^2$-projection of $\grad u_h$ onto the continuous $[\mathbb{P}_1]^d$ space,
the simplest member of the family; it is not the superconvergent patch
recovery of \cite{zz1992a,zz1992b}, and the superscript ``ZZ'' used throughout
is a shorthand for the recovery principle, not for that particular smoothing.} Superconvergence is exactly the property that a distorted mesh
destroys, so $\eta^{\mathrm{ZZ}}$ is expected to be the least robust of the
\michel{four}; anisotropic variants exist \cite{picasso2003}.

\subsection{Bank--Weiser}

This is the closest competitor of the equilibrated flux estimator, and the one
we shall follow quantitatively throughout Section~\ref{sec:numerics}: like the
patchwise reconstruction of Section~\ref{sec:patchwise} it solves small local
problems, and it is the estimator with the reputation of being the sharpest.

The Bank--Weiser estimator \cite{bankweiser1985} solves one local
Neumann-type problem per cell. Let
\begin{equation}\label{eq:bwspace}
  V_K^{\mathrm{BW}} := \{ v \in \mathbb{P}_{k+1}(K) :
                          v(\mathrm{a}) = 0 \ \ \forall \mathrm{a}
                          \in \mathcal{V}_h \cap K \} ,
\end{equation}
which for $k=1$ is the three-dimensional span of the quadratic edge bubbles,
i.e.\ $\mathbb{P}_2(K)$ modulo $\mathbb{P}_1(K)$ in the hierarchical sense;
see \cite{bulle2023} for the systematic construction at arbitrary $k$ and its
\texttt{FEniCS} implementation. Find $e_K^{\mathrm{BW}} \in V_K^{\mathrm{BW}}$
with
\begin{equation}\label{eq:bw}
  \ip{\mu\grad e_K^{\mathrm{BW}}}{\grad v}_K
  = \ip{f + \dive(\mu\grad u_h)}{v}_K
    - \tfrac12\sum_{F\subset\partial K\cap\mathcal{F}_h^i}
      \ip{[\![\mu\grad u_h\cdot\nv]\!]}{v}_F
  \quad \forall v \in V_K^{\mathrm{BW}} ,
\end{equation}
and set $\eta_K^{\mathrm{BW}} := \enorm{e_K^{\mathrm{BW}}}_K$. The right-hand
side is the residual functional $\ell(v) - a(u_h,v)$ written cellwise, each
facet integral being split half and half between its two cells. Four remarks,
which are what the numerical comparison will put figures on.

\begin{enumerate}[label=(\alph*),leftmargin=*,nosep]
  \item \emph{No boundary condition} is imposed on $V_K^{\mathrm{BW}}$: since
        the space contains no constant, the local bilinear form is already
        positive definite on it, and the pure Neumann problem
        \eqref{eq:bw} is solvable without any compatibility condition. This is
        what makes Bank--Weiser cheap, and it is also the reason it needs no
        equilibration step --- and gets no guarantee in exchange.
  \item \emph{Reliability holds only under a saturation assumption} --- that
        the $\mathbb{P}_{k+1}$ solution be strictly closer to $u$ than the
        $\mathbb{P}_k$ one by a fixed factor. The resulting constant is not
        one, it depends on the shape regularity, and, unlike the equilibration
        defect of step~4 of Algorithm~\ref{alg:main}, the assumption cannot be
        checked a posteriori: nothing in the computed output tells the user
        whether the bound holds.
  \item \emph{The $1/2$ split of the facet term is the classical, unweighted
        choice}, and it is not robust with respect to the diffusion
        coefficient: the jump $[\![\mu\grad u_h\cdot\nv]\!]$ carries the
        \emph{larger} of the two values of $\mu$, while the local form on the
        low-coefficient side weighs it against the \emph{smaller} one.
        Coefficient-weighted splits repair this; we use the classical version,
        and Section~\ref{sec:numerics_bw} shows what it costs.
  \item \emph{Cost}: \eqref{eq:bw} is symmetric positive definite and of
        dimension $\dim\mathbb{P}_{k+1}(K) - \#\{\text{vertices}\}$ --- three
        for $k=1$ in 2D --- against a saddle-point problem over a whole vertex
        patch for the equilibrated reconstruction. Per local solve,
        Bank--Weiser is the cheaper of the two.
\end{enumerate}

\subsection{Summary}

Table~\ref{tab:sensitivity} summarises the situation. The point we want to make
is \emph{not} that the equilibrated flux estimator is the sharpest. It is not:
on the \michel{finest} structured \michel{mesh} of Table~\ref{tab:conv} Bank--Weiser overestimates by
$0.4\%$ where the equilibrated flux overestimates by $0.8\%$, and on the
under-resolved boundary layer it is sharper still ($1.74$ against $2.40$). The
point is that it is the only one of the four whose validity does not evaporate
when the mesh is bad --- and, as Section~\ref{sec:numerics_bw} shows, the
sharpness advantage of Bank--Weiser is exactly what distortion takes away.

\section{Goal-oriented estimation}
\label{sec:goal}

Engineers rarely care about the energy norm; they care about a quantity
of interest $J(u)$ --- a displacement at a point, a flux through a surface, an
average stress. Let $J \in V'$ be a bounded linear functional and let $z \in V$
solve the dual problem
\begin{equation}\label{eq:dual}
  a(v,z) = J(v) \qquad \forall v \in V,
\end{equation}
approximated by $z_h \in \Vh$ with $a(v_h,z_h) = J(v_h)$ for all
$v_h \in \Vh$.

\begin{proposition}[Guaranteed goal-oriented bound]\label{prop:goal}
There holds
\begin{equation}\label{eq:goalbound}
  \left| J(u) - J(u_h) \right|
  = \left| a(u-u_h, z-z_h) \right|
  \leq \enorm{u-u_h}\,\enorm{z-z_h}
  \leq \eta^{\mathrm{prim}}\,\eta^{\mathrm{dual}},
\end{equation}
where $\eta^{\mathrm{prim}}$ and $\eta^{\mathrm{dual}}$ are the
equilibrated flux estimators \eqref{eq:upper} of the primal and of the dual
problem.
\end{proposition}

\begin{proof}
$J(u)-J(u_h) = a(u-u_h,z)$ by \eqref{eq:dual}, and $a(u-u_h,z_h)=0$ by
Galerkin orthogonality \eqref{eq:galerkin}, whence
$J(u)-J(u_h)=a(u-u_h,z-z_h)$. Conclude with the Cauchy--Schwarz inequality in
the energy norm, then Theorem~\ref{thm:upper} applied twice.
\end{proof}

Again the constant is one and no mesh assumption is used. Since $a$ is
symmetric, the dual problem is solved with the \emph{same} stiffness matrix as
the primal one --- already factorised --- so the extra cost is one additional
right-hand side plus one additional flux reconstruction, that is, roughly a
doubling of the estimation cost.

Two caveats are in order. First, \eqref{eq:goalbound} bounds
$|J(u)-J(u_h)|$ but gives no information on its \emph{sign}, whereas practice
often calls for a signed correction; sharper, non-product bounds that also
provide a lower bound are available \cite{creuse2017}. Second, the product form
is pessimistic by construction, since it discards the cancellation between the
primal and the dual error: the effectivity index of \eqref{eq:goalbound} grows
like the ratio of the two error distributions. Section~\ref{sec:numerics_goal}
quantifies this, and shows that the resulting index is nevertheless $O(1)$ and,
like the energy bound itself, insensitive to an isolated degenerate cell.


\section{Numerical experiments}
\label{sec:numerics}

All computations are performed with \texttt{FEniCSx} (\texttt{dolfinx}~0.9.0) in
serial. All the scripts, meshes and raw result files are openly
available at
\begin{center}
\url{https://github.com/michelduprez/A-posteriori-error-estimation-for-coarse-or-distorted-meshes}
\end{center}
Every table and figure below is regenerated from that repository by a
single driver, whose \texttt{README} documents the role of each file. Related
open implementations are \texttt{dolfin-tape} \cite{blechta_dolfintape} and the
Bank--Weiser code of \cite{bulle2023}.

\subsection{Protocol}

\paragraph{Test cases.}
In each case the exact solution is known in closed form, so that the true error
$\enorm{u-u_h}$ --- and hence the effectivity index \eqref{eq:Ieff} --- can be
computed and the guarantee $\eta\geq\enorm{u-u_h}$ can be \emph{checked}, not
merely asserted.
\begin{enumerate}[label=(T\arabic*),nosep]
  \item\label{it:T1} \emph{Smooth}: $\Om=(0,1)^2$,
        $u(\xv)=\sin(\pi x_1)\sin(\pi x_2)$,
        $\mu\equiv 1$, $f = 2\mu\pi^2 u$, homogeneous Dirichlet data. Reference
        case: the error is smooth and equidistributed, so any deviation of $I$
        from its structured-mesh value is attributable to the mesh alone.
  \item\label{it:T2} \emph{Singular}: L-shaped domain
        $\Om = (-1,1)^2\setminus([0,1]\times[-1,0])$, $f=0$, and
        $u = r^{2/3}\sin(2\theta/3)$, with $\theta$ measured counter-clockwise
        from the positive $x_1$-axis and running over $(0,3\pi/2)$. Here
        $u\in H^{1+2/3-\epsilon}(\Om)$ only, and the Dirichlet data are
        non-homogeneous on part of $\partial\Om$, so that
        Theorem~\ref{thm:upper} does \emph{not} apply verbatim
        (Remark~\ref{rem:nonhomog}). We run the case nevertheless, with
        $g_h$ the Lagrange interpolant of $g$, precisely in order to see
        whether, and by how much, the bound fails; this is
        Section~\ref{sec:numerics_lshape}. \michel{It} tests the estimator in the presence
        of a corner singularity, where the error is strongly localised.
  \item\label{it:T3} \emph{Boundary layer}: $\Om=(0,1)^2$ and
        \[
          u(\xv) = \big(1-e^{-x_1/\varepsilon}\big)
                   \big(1-e^{-(1-x_1)/\varepsilon}\big)\, x_2(1-x_2),
          \qquad \michel{\varepsilon = 10^{-2},}
        \]
        which vanishes on $\partial\Om$ and exhibits layers of width
        $\varepsilon$ near $x_1=0$ and $x_1=1$; $f$ is obtained by
        differentiation. Natural companion of the anisotropic mesh family
        \ref{it:F5}.
  \item\label{it:T4} \emph{Heterogeneous}: $\mu = 1$ for $x_1 < 1/2$ and
        $\mu = 10^4$ for $x_1 > 1/2$, with the interface aligned with the mesh
        as required by \ref{it:M3}. The interface is placed exactly where
        $\partial u/\partial x_1 = 0$ for $u = \sin(\pi x_1)\sin(\pi x_2)$, so
        that the normal flux $\mu\,\partial u/\partial x_1$ is continuous
        across it and this $u$ \emph{is} the exact solution of the transmission
        problem, with $f = 2\mu\pi^2 u \in L^2(\Om)$. Checks the
        $\mu$-robustness announced in Theorem~\ref{thm:upper}.
\end{enumerate}

\paragraph{Mesh families.}
The synthetic meshes \ref{it:F1}--\ref{it:F5} triangulate the unit square, so
that the closed-form solutions above apply and the effectivity index can be
\emph{measured} rather than estimated; the real meshes \ref{it:F6} and
the L-shaped meshes \ref{it:F7} do not, and are treated separately.
\begin{enumerate}[label=(F\arabic*),nosep]
  \item\label{it:F1} uniform structured meshes, $N=8,\dots,128$ (reference);
  \item\label{it:F2} randomly perturbed meshes: every interior vertex is moved
        by $\delta h\,\boldsymbol{\xi}$, $\boldsymbol{\xi}$ uniform in the unit
        disc, $\delta\in\{0.10,0.20,0.30\}$. Since both endpoints of an edge
        move, the mesh stays untangled as long as $2\delta < 1/\sqrt2$, i.e.\
        $\delta < 0.354$; the code verifies a posteriori that no cell has been
        inverted;
  \item\label{it:F3} \emph{locally damaged} meshes: \emph{one single} interior
        vertex is displaced, all other cells being left untouched, in the
        spirit of \cite{duprez2019damage}. Two variants, controlled by a
        parameter $t\in\{10^{-1},\dots,10^{-4}\}$:
        \emph{needle}, where the vertex is moved a fraction $1-t$ of the way
        toward a neighbour (so $\theta_{\min}\to 0$ while $\theta_{\max}$ stays
        bounded away from $\pi$, and the maximum angle condition still holds);
        and \emph{cap}, where it is moved onto the opposite edge of one of its
        cells (so $\theta_{\max}\to\pi$ and the maximum angle condition fails).
        In both cases $s_K \sim t$;
  \item\label{it:F4} \emph{extent} of the damage: $m\in\{1,4,16,64\}$ vertices
        are cap-damaged at fixed $t=10^{-3}$, either \emph{scattered} over the
        whole domain or grouped in a single contiguous \emph{block}. At equal
        $m$ and equal local quality, the two layouts differ only by the
        diameter of the damaged region, which is precisely the parameter the
        conjecture of Section~\ref{sec:efficiency} points to. Damaged vertices
        are taken on the even sublattice so that no two of them are
        neighbours;
  \item\label{it:F5} globally anisotropic structured meshes with $N\times M$
        cells, $M \in \{N, N/4, N/8, N/16\}$, hence aspect ratios up to
        $\alpha \approx 16$;
  \item\label{it:F6} the \emph{real} meshes shipped with the project, treated
        separately in Section~\ref{sec:numerics_real}: a triangulation of
        $(-1,1)^2$ and four triangulations of a letter glyph, produced by two
        independent meshers;
  \item\label{it:F7} structured triangulations of the L-shaped domain
        of \ref{it:T2}, with $2N\times 2N$ squares of side $1/N$ over the
        enclosing square $(-1,1)^2$, each split into two right triangles and
        the cells of the missing quadrant removed, $N = 4,\dots,32$. These are
        shape-regular meshes: they carry the singularity of \ref{it:T2}, not a
        geometric defect.
\end{enumerate}
Figure~\ref{fig:meshes} shows one representative of each synthetic family.

\begin{figure}[htbp]
\centering
\includegraphics[width=\textwidth]{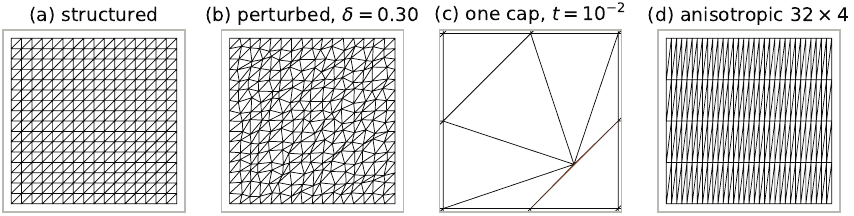}
\caption{The synthetic mesh families. (a) the structured reference
\ref{it:F1}; (b) the random perturbation \ref{it:F2} at the largest admissible
amplitude; (c) a zoom on the single cap-damaged vertex of \ref{it:F3}, the
flattened cell being highlighted --- every other cell of that mesh is the
structured one of panel (a); (d) the anisotropic family \ref{it:F5}.}
\label{fig:meshes}
\end{figure}

\paragraph{Quantities reported.}
For each pair (test case, mesh) we report the true energy error
$\enorm{u-u_h}$, the four estimators, and their global effectivity indices
\eqref{eq:Ieff}. The equilibrated flux is computed with the
$\BDM_1\times\mathbb{P}_0$ pair of Section~\ref{sec:mixed}; the residual
estimator uses the engineering calibration $C_1=C_2=1$, and $G_h$ in
\eqref{eq:zz} is the global $L^2$-projection onto the continuous
$[\mathbb{P}_1]^2$ space. In addition we report
\begin{itemize}[nosep]
  \item the \emph{certificate} \[\max_K |\ip{\dive\sig_h - f}{1}_K| /
        \michel{\max\big\{\ip{|f|}{1}_K,\ \ip{|\dive\sig_h|}{1}_K,\
        h_K^{-1}\ip{|\sig_h|}{1}_K\big\}},\] which must be at round-off level
        (step 4 of Algorithm~\ref{alg:main})\michel{; the denominator is the
        size of the terms that must cancel, and does not degenerate when
        $f = 0$ as in test case \ref{it:T2}}. \michel{It never exceeded
        $7.2\cdot 10^{-10}$ over the whole campaign, and $2.3\cdot 10^{-9}$
        if the dual problems of Section~\ref{sec:numerics_goal} are
        included};
  \item the worst \emph{cellwise} ratio $\eta_K / \enorm{u-u_h}_K$ over the
        badly shaped cells ($s_K < 0.1$), written $I_K^{\mathrm{bad}}$. Note
        that Theorem~\ref{thm:efficiency} bounds $\eta_K$ by the error on the
        \emph{patch} $\omega_K$, so a cellwise ratio slightly above one is
        expected and is not a violation of anything; the quantity is reported
        as a proxy for the local sharpness on the damaged cells.
\end{itemize}

\subsection{How it is computed: a reference implementation}
\label{sec:code}

The listing below computes the equilibrated flux by the global mixed
formulation \eqref{eq:mixedpb} and evaluates \emph{both} indicators
\zuqi{\eqref{eq:etaR} and \eqref{eq:etaF}}, including the residual term that is often
(wrongly) omitted, as well as the parameter $\mu$. It is a complete, runnable
\texttt{FEniCSx} script; the output it produces is quoted after the listing.

\begin{lstlisting}[style=py, linerange=LISTING-ENDLISTING,
caption={Guaranteed error bound
via a global $\BDM_1\times\mathbb{P}_0$ flux reconstruction. This is
\texttt{code/fenicsx/minimal\_example.py} included verbatim, so the listing
and the script that produced the output below cannot drift apart.},
label={lst:fenicsx}]
"""
Self-contained minimal example: guaranteed a posteriori error bound for the
Poisson problem, by a global mixed H(div) flux reconstruction.

This is the script reproduced in the listing in the paper.
Run inside a dolfinx container (tested with dolfinx 0.9.0), in serial:

    python3 minimal_example.py

The paper includes everything below the LISTING marker, so that the listing and
the script that produced its output cannot drift apart.
"""

# LISTING
from mpi4py import MPI
import numpy as np
import ufl
import basix.ufl
from dolfinx import fem, mesh as dmesh
from dolfinx.fem.petsc import LinearProblem

LU = {"ksp_type": "preonly", "pc_type": "lu",
      "pc_factor_mat_solver_type": "mumps"}
QD = {"quadrature_degree": 8}

# ------------------------------------------------------------ data ----
N = 64
mu = 1.0                            # physical parameter, tracked
msh = dmesh.create_unit_square(MPI.COMM_WORLD, N, N,
                              dmesh.CellType.triangle)

x = ufl.SpatialCoordinate(msh)
u_ex = ufl.sin(ufl.pi * x[0]) * ufl.sin(ufl.pi * x[1])
f = -ufl.div(mu * ufl.grad(u_ex))          # = 2 mu pi^2 u_ex

# -------------------------------------------------- primal problem ----
V = fem.functionspace(msh, ("Lagrange", 1))
msh.topology.create_connectivity(1, 2)
bdofs = fem.locate_dofs_topological(
    V, 1, dmesh.exterior_facet_indices(msh.topology))
bc = fem.dirichletbc(fem.Constant(msh, 0.0), bdofs, V)

u, v = ufl.TrialFunction(V), ufl.TestFunction(V)
uh = LinearProblem(ufl.inner(mu * ufl.grad(u), ufl.grad(v)) * ufl.dx(metadata=QD),
                   ufl.inner(f, v) * ufl.dx(metadata=QD),
                   bcs=[bc], petsc_options=LU).solve()

# ------------------------- equilibrated flux, global mixed ----
# Inf-sup stable pair BDM_{p+1} x P_p: div sigma_h = Pi_p f exactly.
p = 0
W = fem.functionspace(msh, basix.ufl.mixed_element([
    basix.ufl.element("BDM", msh.basix_cell(), p + 1),
    basix.ufl.element("DG", msh.basix_cell(), p)]))

(sig, lam) = ufl.TrialFunctions(W)
(tau, q) = ufl.TestFunctions(W)

#  ( mu^{-1} sigma_h , tau ) - ( lambda_h , div tau ) = 0
#  ( div sigma_h , q )                                = ( f , q )
# u = 0 is *natural* here: no essential boundary condition.
a = ((1.0 / mu) * ufl.inner(sig, tau) - lam * ufl.div(tau)
     + ufl.div(sig) * q) * ufl.dx(metadata=QD)
# f integrated with the same rule as in the certificate below: tested with
# q = 1_K, the second equation is then exactly the equilibration (EF).
L = ufl.inner(f, q) * ufl.dx(metadata=QD)
w = LinearProblem(a, L, bcs=[], petsc_options=LU).solve()
sig_h = w.sub(0)

# --------------------------------------------- cellwise indicators ----
DG0 = fem.functionspace(msh, ("DG", 0))
v0 = ufl.TestFunction(DG0)
dx = ufl.dx(metadata=QD)
h_K = ufl.CellDiameter(msh)


def cellwise(expr):
    return np.asarray(fem.assemble_vector(fem.form(expr * v0 * dx)).array)


r = sig_h + mu * ufl.grad(uh)
etaF = np.sqrt(cellwise((1.0 / mu) * ufl.inner(r, r)))       # eta_F
etaR = np.sqrt(cellwise((h_K / np.pi) ** 2 * (1.0 / mu)
                        * (f - ufl.div(sig_h)) ** 2))        # eta_R
eta = np.sqrt(((etaR + etaF) ** 2).sum())                    # bound

# ------------------------------------------------ check the bound ----
e = ufl.grad(uh - u_ex)
err = np.sqrt(fem.assemble_scalar(fem.form(ufl.inner(mu * e, e) * dx)))

# certificate: relative cellwise equilibration defect.  The scale is the
# largest of the terms that must cancel, so that the test stays meaningful
# when |f| is large -- and defined when f vanishes.
scale = np.maximum.reduce([cellwise(abs(f)), cellwise(abs(ufl.div(sig_h))),
                           cellwise(ufl.sqrt(ufl.inner(sig_h, sig_h)) / h_K)])
defect = np.abs(cellwise(ufl.div(sig_h) - f)) / scale

print(f"energy error   |||u-u_h|||   = {err:.6e}")
print(f"estimator      eta           = {eta:.6e}")
print(f"  flux part    ||etaF||      = {np.linalg.norm(etaF):.6e}")
print(f"  oscillation  ||etaR||      = {np.linalg.norm(etaR):.6e}")
print(f"effectivity    I = eta/err   = {eta / err:.4f}")
print(f"max rel. equilibration defect= {defect.max():.2e}")
assert eta >= err, "the guaranteed bound is violated -- check the code!"

# ENDLISTING
\end{lstlisting}

\noindent
Its output on the $64\times 64$ mesh is
\begin{verbatim}
energy error   |||u-u_h|||   = 5.451370e-02
estimator      eta           = 5.539722e-02
  flux part    ||etaF||      = 5.450766e-02
  oscillation  ||etaR||      = 1.135797e-03
effectivity    I = eta/err   = 1.0162
max rel. equilibration defect= 1.03e-12
\end{verbatim}
Three things are worth reading off these six lines. The bound holds with an
effectivity of $1.6\%$; the oscillation term \eqref{eq:osc} accounts for $2\%$
of the estimator, so it is small but \emph{not} negligible --- dropping it
would give $\eta = 5.4508\cdot 10^{-2} < \enorm{u-u_h} = 5.4514\cdot 10^{-2}$
and the assertion on the last line would fire, which is exactly the trap of
Remark~\ref{rem:dropR}; and the equilibration certificate is at round-off, so
\eqref{eq:EF} is verified to hold, and not merely assumed to.

The assertion on the last line is the central claim of the paper: it must never
fire, whatever the mesh. Over the $32$ runs of the campaign that are
covered by Theorem~\ref{thm:upper}, and the one real mesh with a known
solution, it never did --- for either reconstruction. The patchwise variant of
Section~\ref{sec:patchwise} replaces the global solve by $\#\mathcal{V}_h$
local ones; it is implemented in \texttt{patchwise.py} and reported in
Section~\ref{sec:numerics_pw}.

\subsection{Reference case: convergence and asymptotic behaviour}
\label{sec:numerics_ref}

Table~\ref{tab:conv} reports the reference family \ref{it:F1} on the smooth
test case \ref{it:T1}.

\begin{table}[ht]
\centering
\begin{tabular}{lrrrrrrr}
\toprule
mesh & cells & $\enorm{u-u_h}$ & $\eta^{\mathrm{eq}}$ & $I^{\mathrm{eq}}$ & $I^{\mathrm{BW}}$ & $I^{\mathrm{res}}$ & $I^{\mathrm{ZZ}}$ \\
\midrule
F1 $N=8$ & 128 & 4.318e-01 & 4.878e-01 & 1.130 & 1.082 & 5.526 & 0.960 \\
F1 $N=16$ & 512 & 2.175e-01 & 2.316e-01 & 1.065 & 1.037 & 5.601 & 0.984 \\
F1 $N=32$ & 2048 & 1.090e-01 & 1.125e-01 & 1.032 & 1.017 & 5.632 & 0.993 \\
F1 $N=64$ & 8192 & 5.451e-02 & 5.540e-02 & 1.016 & 1.008 & 5.645 & 0.997 \\
F1 $N=128$ & 32768 & 2.726e-02 & 2.748e-02 & 1.008 & 1.004 & 5.651 & 0.998 \\
\bottomrule
\end{tabular}
\caption{Test case \ref{it:T1} on the structured family \ref{it:F1}, $k=1$.
The equilibrated flux estimator is \emph{asymptotically exact}
($I^{\mathrm{eq}}\to 1$), the residual estimator with $C=1$ overestimates by a
factor close to $5.65$, and the recovery estimator converges to $1$
\emph{from below} --- it underestimates the error on every mesh of this
family. Bank--Weiser is the sharpest of the four here, and stays above one.}
\label{tab:conv}
\end{table}

Four observations, all of which persist throughout the campaign \michel{on
the unit square}.
\begin{enumerate}[nosep]
  \item $I^{\mathrm{eq}}$ decreases monotonically from $1.130$ to $1.008$: the
        guaranteed bound is not merely valid, it becomes asymptotically exact.
        The residual indicator $\eta^{\mathrm{R}}$ contributes only
        $\norm{\eta^{\mathrm{R}}}/\norm{\eta^{\mathrm{F}}} \approx 2\%$ of the
        estimator at $N=64$ and decreases at the higher rate predicted by
        Proposition~\ref{prop:oscillation}.
  \item $I^{\mathrm{res}} \approx 5.6$: the explicit residual estimator with
        the naive calibration $C=1$ is off by more than half an order of
        magnitude. This is not a defect of the estimator but a reminder that
        its constant is unknown.
  \item $I^{\mathrm{ZZ}} < 1$ on every mesh. The recovery estimator
        \emph{underestimates} the error --- benignly here ($0.96$ to $0.998$),
        much less benignly below.
  \item $I^{\mathrm{BW}}$ is the sharpest of the four, $1.082$ down to
        $1.004$: on a structured mesh Bank--Weiser overestimates by about half
        as much as the equilibrated flux. Section~\ref{sec:numerics_bw} follows
        what becomes of that advantage.
\end{enumerate}

\subsection{Distorted meshes: the main experiment}
\label{sec:numerics_main}

Table~\ref{tab:main} is the core of the paper: the same test case \ref{it:T1}
computed on meshes of steadily degrading quality.

\begin{table}[ht]
\centering
\small
\begin{tabular}{lrrrrrrr}
\toprule
mesh & $s_{\min}$ & $\theta_{\max}$ (deg) & $I^{\mathrm{eq}}$ & $I^{\mathrm{BW}}$ & $I^{\mathrm{res}}$ & $I^{\mathrm{ZZ}}$ & $I_K^{\mathrm{bad}}$ \\
\midrule
F2 $\delta=0.10$ & 5.6e-01 & 106.9 & 1.033 & 1.022 & 5.619 & 0.989 & --- \\
F2 $\delta=0.20$ & 4.0e-01 & 126.1 & 1.033 & 1.036 & 5.583 & 0.978 & --- \\
F2 $\delta=0.30$ & 2.3e-01 & 147.3 & 1.034 & 1.062 & 5.543 & 0.961 & --- \\
F3 needle $t=10^{-1}$ & 5.3e-02 & 132.0 & 1.032 & 1.021 & 5.637 & 0.992 & 1.00 \\
F3 needle $t=10^{-2}$ & 5.0e-03 & 134.7 & 1.032 & 1.048 & 5.639 & 0.991 & 1.00 \\
F3 needle $t=10^{-3}$ & 5.0e-04 & 135.0 & 1.032 & 1.282 & 5.640 & 0.991 & 1.00 \\
F3 needle $t=10^{-4}$ & 5.0e-05 & 135.0 & 1.032 & 2.664 & 5.640 & 0.991 & 1.00 \\
F3 cap $t=10^{-1}$ & 1.0e-01 & 168.6 & 1.032 & 1.023 & 5.631 & 0.992 & 1.00 \\
F3 cap $t=10^{-2}$ & 1.0e-02 & 178.9 & 1.032 & 1.028 & 5.636 & 0.992 & 1.00 \\
F3 cap $t=10^{-3}$ & 1.0e-03 & 179.9 & 1.032 & 1.028 & 5.637 & 0.992 & 1.00 \\
F3 cap $t=10^{-4}$ & 1.0e-04 & 180.0 & 1.032 & 1.028 & 5.637 & 0.992 & 1.00 \\
F4 scattered $m=1$ & 1.0e-03 & 179.9 & 1.032 & 1.019 & 5.639 & 0.993 & 1.00 \\
F4 block $m=1$ & 1.0e-03 & 179.9 & 1.032 & 1.028 & 5.637 & 0.992 & 1.00 \\
F4 scattered $m=4$ & 1.0e-03 & 179.9 & 1.032 & 1.032 & 5.639 & 0.991 & 1.06 \\
F4 block $m=4$ & 1.0e-03 & 179.9 & 1.031 & 1.058 & 5.651 & 0.988 & 1.00 \\
F4 scattered $m=16$ & 1.0e-03 & 179.9 & 1.031 & 1.095 & 5.685 & 0.983 & 1.06 \\
F4 block $m=16$ & 1.0e-03 & 179.9 & 1.030 & 1.152 & 5.705 & 0.980 & 1.00 \\
F4 scattered $m=64$ & 1.0e-03 & 179.9 & 1.031 & 1.198 & 5.750 & 0.964 & 1.14 \\
F4 block $m=64$ & 1.0e-03 & 179.9 & 1.028 & 1.321 & 5.814 & 0.964 & 1.02 \\
F5 $32\times32$ & 7.1e-01 & 90.0 & 1.032 & 1.017 & 5.632 & 0.993 & --- \\
F5 $32\times8$ & 2.4e-01 & 90.0 & 1.092 & 1.163 & 6.962 & 0.985 & --- \\
F5 $32\times4$ & 1.2e-01 & 90.0 & 1.182 & 1.529 & 8.648 & 0.965 & --- \\
F5 $32\times2$ & 6.2e-02 & 90.0 & 1.401 & 2.297 & 10.429 & 0.880 & 1.87 \\
\bottomrule
\end{tabular}
\caption{Test case \ref{it:T1}, $N=32$, $k=1$, on the mesh families
\ref{it:F2}--\ref{it:F5}. The scaled Jacobian $s_{\min}$ degrades over four
orders of magnitude and $\theta_{\max}$ reaches $180$ degrees, yet
$I^{\mathrm{eq}}$ never leaves the interval $[1.028, 1.401]$ and is never below
one. $I^{\mathrm{ZZ}}$, by contrast, is below one in every single row, and
$I^{\mathrm{BW}}$ --- the sharpest of the four on the structured mesh ---
reaches $2.664$ on the thinnest needle.}
\label{tab:main}
\end{table}

Figure~\ref{fig:effectivity} plots the same information for four of
the families. The reading is immediate: the two equilibrated curves
stay just above the guarantee line whatever is done to the mesh, the
global reconstruction being lifted only by the anisotropy of panel~(c) and the
patchwise one only by the accumulated damage of panel~(d); the recovery curve
lives in the shaded region, where an estimator reports less error than has been
committed.

\begin{figure}[htbp]
\centering
\includegraphics[width=\textwidth]{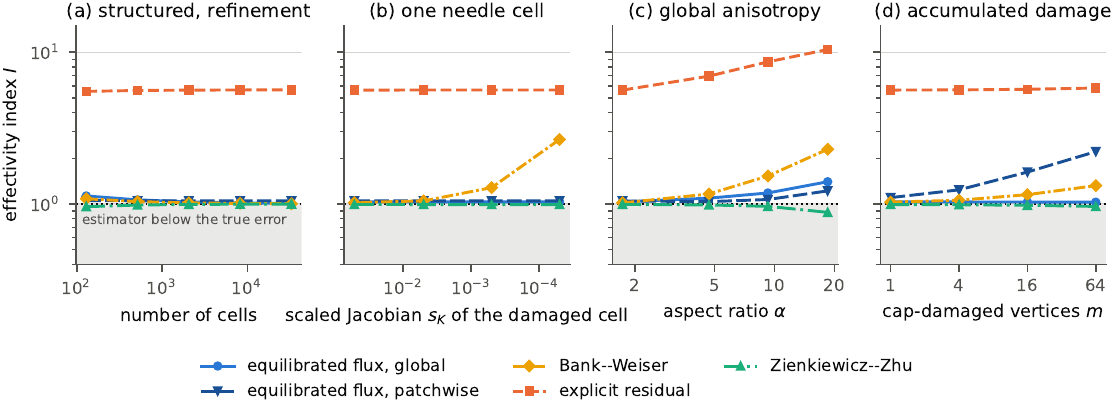}
\caption{Effectivity indices, test case \ref{it:T1}, for the two
reconstructions of the guaranteed bound and the three competitors.
(a) uniform refinement of the structured mesh; (b) a single needle cell whose
scaled Jacobian is swept over four decades, the rest of the mesh being
untouched; (c) increasing global anisotropy; (d) an increasing number
$m$ of cap-damaged vertices grouped in a contiguous block, at fixed local
quality $s_K = 10^{-3}$. In the shaded band the estimator is \emph{below} the
true error: neither reconstruction of the guaranteed bound ever enters
it, the Zienkiewicz--Zhu estimator never leaves it. Panel~(b) is the sharpest
contrast of the campaign: one degenerate cell out of $2048$ leaves the
equilibrated index at $1.032$ and multiplies the Bank--Weiser index by $2.6$.
Panels~(c) and (d) are mirror images of one another: anisotropy
degrades the global reconstruction and leaves the patchwise one flat,
accumulated damage does the reverse (Section~\ref{sec:numerics_pw}).}
\label{fig:effectivity}
\end{figure}

\subsection{Randomly perturbed meshes}
\label{sec:numerics_perturbed}

The first three rows of Table~\ref{tab:main} perturb every interior vertex. As
$\delta$ grows from $0.10$ to $0.30$, the worst scaled Jacobian falls from
$0.56$ to $0.23$ and the largest angle climbs from $107$ to $147$ degrees, i.e.\
the mesh becomes one that most quality-control tools would reject. The
equilibrated effectivity index moves from $1.033$ to $1.034$: it is, for
practical purposes, \emph{constant}. The recovery estimator, on the other hand,
drifts from $0.989$ to $0.961$: the perturbation destroys the superconvergence
on which its accuracy rests, and it does so in the unsafe direction.
Bank--Weiser degrades mildly and safely, $1.022$ to $1.062$, losing its
advantage over the equilibrated flux by $\delta = 0.30$ without ever falling
below one.

\subsection{Locally damaged meshes, and the extent of the damage}
\label{sec:numerics_damaged}

The next eight rows of Table~\ref{tab:main} displace a \emph{single} vertex.
Between $t=10^{-1}$ and $t=10^{-4}$ the scaled Jacobian of the damaged cell
drops by three orders of magnitude, down to $s_K = 5\cdot 10^{-5}$ for the
needle and $10^{-4}$ for the cap, and $\theta_{\max}$ reaches $180.0$ degrees
to the printed precision. The response of the estimator is nil:
$I^{\mathrm{eq}} = 1.032$ throughout, to four digits, and the \emph{cellwise}
sharpness on the damaged cell itself, $I_K^{\mathrm{bad}}$, stays at $1.00$.
An isolated degenerate cell --- of either type, needle or cap --- is simply
invisible to the equilibrated flux estimator. This is the numerical
counterpart of Theorem~\ref{thm:upper}: nothing in its proof could have
noticed. It is \emph{not} invisible to Bank--Weiser, whose index reaches
$2.664$ on the thinnest needle; we come back to this in
Section~\ref{sec:numerics_bw}, since it is the clearest difference the two
local-problem estimators exhibit.

Table~\ref{tab:damaged} isolates the effect of the \emph{extent} of the damaged
region, family \ref{it:F4}: the number $m$ of degenerate cells and their local
quality are held fixed, only their spatial arrangement changes.

\begin{table}[ht]
\centering
\begin{tabular}{lrrrrr}
\toprule
mesh & $m$ & $\operatorname{diam}/h$ & $I^{\mathrm{eq}}$ & $I^{\mathrm{BW}}$ & $I_K^{\mathrm{bad}}$ \\
\midrule
F4 scattered $m=1$ & 1 & 0.0 & 1.0323 & 1.019 & 1.00 \\
F4 block $m=1$ & 1 & 0.0 & 1.0321 & 1.028 & 1.00 \\
F4 scattered $m=4$ & 4 & 27.9 & 1.0321 & 1.032 & 1.06 \\
F4 block $m=4$ & 4 & 2.8 & 1.0314 & 1.058 & 1.00 \\
F4 scattered $m=16$ & 16 & 34.4 & 1.0311 & 1.095 & 1.06 \\
F4 block $m=16$ & 16 & 8.9 & 1.0297 & 1.152 & 1.00 \\
F4 scattered $m=64$ & 64 & 36.9 & 1.0307 & 1.198 & 1.14 \\
F4 block $m=64$ & 64 & 19.8 & 1.0278 & 1.321 & 1.02 \\
\bottomrule
\end{tabular}
\caption{Family \ref{it:F4}: $m$ cap-damaged vertices at fixed $t=10^{-3}$
(hence fixed $s_K=10^{-3}$ and $\theta_{\max}=179.9$ degrees), test case
\ref{it:T1}, $N=32$. Only the diameter of the damaged region changes between
the \emph{scattered} and the \emph{block} rows at equal $m$. Note that
$I^{\mathrm{BW}}$ responds to the accumulation of damage where
$I^{\mathrm{eq}}$ does not.}
\label{tab:damaged}
\end{table}

Two conclusions, one of which is negative and worth stating as such.
\begin{enumerate}[nosep]
  \item The global index is remarkably stable, $1.0278 \le I^{\mathrm{eq}} \le
        1.0323$, and if anything \emph{decreases} as damage accumulates: a
        damaged mesh is a slightly finer mesh, and the estimator tracks the
        error it actually commits.
  \item With the global reconstruction, the experiment does \emph{not}
        discriminate the conjecture of Section~\ref{sec:efficiency}. At equal
        $m$, the contiguous block (diameter $\michel{19.8h}$ at $m=64$) behaves
        marginally \emph{better} than the scattered layout (diameter $\michel{36.9h}$),
        $I_K^{\mathrm{bad}} = 1.02$ against $1.14$, and the differences are
        within the variation caused by where in $\Om$ the damaged cells happen
        to fall, since $\enorm{u-u_h}_K$ itself varies by an order of magnitude
        across the domain. This is a limitation of the reconstruction, not of
        the mesh family: the conjecture is about the local efficiency constant
        $C_{\mathrm{st}}(\gamma_{\omega_K})$, and the global reconstruction
        exposes no local constant. Section~\ref{sec:numerics_pw} redoes
        the experiment with the patchwise reconstruction, which does, and
        settles the question --- against the conjecture as we stated it.
\end{enumerate}

\subsection{Anisotropic meshes}
\label{sec:numerics_aniso}

The last four rows of Table~\ref{tab:main} are the only ones on which the
equilibrated estimator visibly degrades: $I^{\mathrm{eq}}$ grows from $1.032$
to $1.401$ as the cells are stretched from aspect ratio $1$ to about $16$, and
the cellwise sharpness on the worst cells reaches $I_K^{\mathrm{bad}} = 1.87$.
This is consistent with Theorem~\ref{thm:efficiency} and with the distinction
we have insisted on: anisotropy is a \emph{global} loss of shape regularity ---
every cell is bad --- whereas the damage of family \ref{it:F3} is local. The
guarantee is of course untouched; what is lost is sharpness, and it is lost
gradually. For comparison, $I^{\mathrm{res}}$ nearly doubles over the same
range, from $5.63$ to $10.43$, which is the practical statement that no fixed
calibration constant $C$ can serve both a structured and an anisotropic mesh.
Bank--Weiser degrades faster than either, $1.017$ to $2.297$, crossing the
equilibrated flux between $\alpha = 1$ and $\alpha = 4$.

\subsection{The two reconstructions of the same bound}
\label{sec:numerics_pw}

Everything reported so far uses the global mixed reconstruction of
Section~\ref{sec:mixed}. Table~\ref{tab:patchwise} adds the patchwise one of
Section~\ref{sec:patchwise}, at the degree $p=k=1$ prescribed by
\cite{ernvohralik2015} and at the cheapest admissible degree $p=0$. The two
reconstructions certify the same quantity with the same constant; they differ
only in sharpness, and the comparison is instructive because they do \emph{not}
lose sharpness under the same defects.

\begin{table}[ht]
\centering
\begin{tabular}{lrrrr}
\toprule
mesh & $I^{\mathrm{eq}}$ & $I^{\mathrm{pw}}$ ($p=1$) & $I^{\mathrm{pw}}$ ($p=0$) & $I^{\mathrm{BW}}$ \\
\midrule
F1 $N=8$ & 1.130 & 1.053 & 1.295 & 1.082 \\
F1 $N=16$ & 1.065 & 1.049 & 1.240 & 1.037 \\
F1 $N=32$ & 1.032 & 1.048 & 1.226 & 1.017 \\
F1 $N=64$ & 1.016 & 1.047 & 1.247 & 1.008 \\
F1 $N=128$ & 1.008 & 1.047 & 1.309 & 1.004 \\
F2 $\delta=0.30$ & 1.034 & 1.055 & 1.790 & 1.062 \\
F3 needle $t=10^{-4}$ & 1.032 & 1.048 & 1.226 & 2.664 \\
F3 cap $t=10^{-4}$ & 1.032 & 1.112 & 1.238 & 1.028 \\
F4 scattered $m=64$ & 1.031 & 1.812 & 3.373 & 1.198 \\
F4 block $m=64$ & 1.028 & 2.215 & 2.864 & 1.321 \\
F5 $32\times32$ & 1.032 & 1.048 & 1.226 & 1.017 \\
F5 $32\times8$ & 1.092 & 1.036 & 1.253 & 1.163 \\
F5 $32\times4$ & 1.182 & 1.070 & 1.357 & 1.529 \\
F5 $32\times2$ & 1.401 & 1.220 & 1.638 & 2.297 \\
\bottomrule
\end{tabular}
\caption{Effectivity indices of the two reconstructions of the
guaranteed bound, test case \ref{it:T1}. $I^{\mathrm{eq}}$ is the global mixed
reconstruction of Section~\ref{sec:mixed}, $I^{\mathrm{pw}}$ the patchwise one
of Section~\ref{sec:patchwise}. The global one degrades under global anisotropy
and ignores local damage; the patchwise one does the opposite. The $p=0$ column
shows that the degree of the local space is not a free parameter.}
\label{tab:patchwise}
\end{table}

Three observations.
\begin{enumerate}[nosep]
  \item \textbf{The degree $p$ is not free.} At $p=1$ the patchwise
        index is flat at $1.047$--$1.053$ over the whole structured family. At
        $p=0$ it is not merely larger ($1.23$ to $1.31$) but \emph{non-monotone}
        under refinement, so that the reconstruction no longer converges at the
        rate of the error. The bound of Theorem~\ref{thm:upper} holds for either
        --- any admissible $\sig_h$ gives a valid certificate --- but only
        $p=k$ delivers a sharp one. This is the numerical face of the
        $p$-robustness theory of \cite{braess2009,ernvohralik2020}, and it is
        worth stating explicitly, since $p=0$ is the natural choice for a
        practitioner who reads only the reliability statement.
  \item \textbf{Anisotropy costs the global reconstruction, not the
        patchwise one.} Over the family \ref{it:F5}, $I^{\mathrm{eq}}$ grows
        from $1.032$ to $1.401$ while $I^{\mathrm{pw}}$ stays between $1.036$
        and $1.220$; at aspect ratio $4$ the patchwise index is actually
        \emph{smaller} than on the isotropic mesh. The same holds on the
        under-resolved boundary layer, where $I^{\mathrm{pw}} = 1.583$ against
        $I^{\mathrm{eq}} = 2.402$ (Table~\ref{tab:cases}).
  \item \textbf{Accumulated damage costs the patchwise reconstruction,
        not the global one.} The two indices exchange roles on
        family~\ref{it:F4}: at $m=64$ cap-damaged vertices in a contiguous
        block, $I^{\mathrm{eq}} = 1.028$ while $I^{\mathrm{pw}} = 2.215$.
        Notably, a \emph{needle} leaves the patchwise index untouched
        ($1.048$) while a \emph{cap} of the same scaled Jacobian moves it to
        $1.112$ --- the reverse of the Bank--Weiser behaviour of
        Section~\ref{sec:numerics_bw}, which is blind to caps and sensitive to
        needles.
\end{enumerate}

Panels~(c) and (d) of Figure~\ref{fig:effectivity} display the exchange
directly: the two curves cross over from one panel to the other, while the
three competitors behave the same way in both.

The practical reading is that the guarantee is a property of the
\emph{framework} and the sharpness a property of the \emph{reconstruction}, and
that no single reconstruction is uniformly sharpest. Since both certificates
are valid, the sharper of the two may simply be retained: the minimum of two
guaranteed bounds is a guaranteed bound. Over the whole campaign that minimum
never exceeds $\michel{1.22}$ on the smooth test case.

\paragraph{The damage-extent conjecture.}
Section~\ref{sec:numerics_damaged} could not discriminate the
conjecture of Section~\ref{sec:efficiency} because the global reconstruction
exposes no local efficiency constant. The patchwise one does, and it answers
the question --- in the negative, as stated. At equal $m$ and equal local
quality, the \emph{block} layout is systematically worse than the
\emph{scattered} one: $1.239$ against $1.123$ at $m=4$, $1.624$ against $1.407$
at $m=16$, $2.215$ against $1.812$ at $m=64$. Yet the block has the
\emph{smaller} diameter ($\michel{19.8}\,h$ against $\michel{36.9}\,h$ at $m=64$,
Table~\ref{tab:damaged}). The diameter of the damaged region is therefore not
the relevant parameter; what degrades $C_{\mathrm{st}}$ is the number of
flattened cells sharing a single vertex patch, which is exactly the quantity
that enters the local efficiency proof and which the block layout maximises at
fixed $m$. We record this as the corrected form of the conjecture.

\subsection{Does any quality measure predict the effectivity?}
\label{sec:numerics_corr}

Section~\ref{sec:quality} introduced the quality vector $\mathbf{q}_K$
of \eqref{eq:qualityvector} and announced that it would be correlated with the
local effectivity index. Table~\ref{tab:corr} reports the empirical Pearson
correlation $\rho(I_K,\cdot)$ between each component and the local index
\michel{$I_K$ of \eqref{eq:localeff},}
computed with the patchwise reconstruction --- the denominator being
the patch error, which is what Theorem~\ref{thm:efficiency} actually bounds.

\begin{table}[ht]
\centering
\small
\begin{tabular}{lrrrrrr}
\toprule
mesh & $\rho(I_K, \gamma_K)$ & $\rho(I_K, \alpha_K)$ & $\rho(I_K, s_K)$ & $\rho(I_K, \kappa_K)$ & $\rho(I_K, \theta_{\min})$ & $\rho(I_K, \theta_{\max})$ \\
\midrule
F2 $\delta=0.10$ & -0.03 & -0.04 & 0.08 & -0.04 & 0.08 & -0.01 \\
F2 $\delta=0.20$ & -0.07 & -0.08 & 0.11 & -0.09 & 0.10 & -0.02 \\
F2 $\delta=0.30$ & -0.11 & -0.13 & 0.14 & -0.14 & 0.12 & -0.03 \\
F3 needle $t=10^{-4}$ & -0.25 & -0.25 & 0.16 & -0.25 & 0.15 & -0.04 \\
F3 cap $t=10^{-4}$ & 0.01 & 0.01 & -0.32 & 0.01 & -0.33 & 0.27 \\
F4 scattered $m=64$ & 0.12 & 0.12 & -0.36 & 0.12 & -0.36 & 0.30 \\
F4 block $m=64$ & 0.19 & 0.19 & -0.48 & 0.19 & -0.47 & 0.33 \\
\bottomrule
\end{tabular}
\caption{Correlation between the local effectivity index
\michel{\eqref{eq:localeff}} and each component of the quality vector
\eqref{eq:qualityvector}, test case \ref{it:T1}, $N=32$. No component reaches
$|\rho| = \michel{0.5}$, and the sign of $\rho(I_K,\gamma_K)$ flips between the needle
and the cap families.}
\label{tab:corr}
\end{table}

The outcome is a negative result, and we regard it as the useful
one.
\begin{itemize}[nosep]
  \item \emph{No measure predicts the effectivity.} The largest
        correlation over the whole table is $|\rho| = \michel{0.48}$, for $s_K$ on the
        block-damaged mesh. A quality report is not a proxy for an error bound,
        which is the practical form of the message of
        Table~\ref{tab:sensitivity}.
  \item \emph{The sign is not even stable.} $\rho(I_K,\gamma_K)$ is
        $-0.25$ on the needle family and $+0.19$ on the block-damaged one. A
        shape-regularity ratio that grows can accompany either an increase or a
        decrease of the local effectivity, depending on \emph{how} the cell
        degenerates. This is Babu\v{s}ka and Aziz \cite{babuskaaziz1976} seen
        from the a posteriori side.
  \item \emph{Three of the six measures carry the same information.}
        The columns $\gamma_K$, $\alpha_K$ and $\kappa_K$ agree to two decimals
        \michel{on the damaged families, and to within $0.03$ on the perturbed
        ones}. Two reasons combine: for a triangle
        $\kappa_K = (\alpha_K + \alpha_K^{-1})/2$ is an exact function of
        $\alpha_K$; and on a mesh in which all cells but a few carry the same
        quality value, the correlation coefficient is determined by \emph{which}
        cells are outliers and not by how large the outlying values are. The
        measures that do separate the families are $s_K$ and the extremal
        angles.
\end{itemize}

\subsection{Bank--Weiser: sharper on good meshes, more fragile on bad ones}
\label{sec:numerics_bw}

The $I^{\mathrm{BW}}$ column of Tables~\ref{tab:conv}--\ref{tab:cases} deserves
to be read on its own, because it is the only competitor that is a serious one.

\paragraph{On good meshes it wins.}
On the structured family, $I^{\mathrm{BW}}$ falls from $1.082$ to $1.004$ while
$I^{\mathrm{eq}}$ falls from $1.130$ to $1.008$: Bank--Weiser overestimates by
about half as much, at every refinement level, and both are asymptotically
exact. On the under-resolved boundary layer \ref{it:T3} the gap is wider still,
$1.743$ against $2.402$. If sharpness on a reasonable mesh were the only
criterion, Bank--Weiser would be the estimator to use --- which is presumably
why it is.

\paragraph{A single needle cell costs it a factor of $2.6$.}
The needle rows of Table~\ref{tab:main} are the sharpest contrast of the whole
campaign. As $t$ goes from $10^{-1}$ to $10^{-4}$, $I^{\mathrm{eq}}$ stays at
$1.032$ to four digits, while $I^{\mathrm{BW}}$ climbs
$1.021 \to 1.048 \to 1.282 \to 2.664$. One cell out of $2048$, and the
Bank--Weiser estimator has lost a factor of $2.6$; the equilibrated one has not
moved. Figure~\ref{fig:effectivity}(b) plots exactly this.

The \emph{cap} rows show nothing of the sort:
$I^{\mathrm{BW}}$ moves from $1.023$ to $1.028$ and then stops. The two damages have \michel{comparable}
scaled Jacobian\michel{s (within a factor of two at equal $t$)} and the same shape-regularity ratio $\gamma_K \approx 2000$, so
$\gamma_K$ does not explain the difference. What distinguishes them is that the
needle collapses an \emph{edge} --- two vertices come to within $t\,h$ of each
other --- whereas the cap collapses only the \emph{area}, all three edges
keeping length $O(h)$. The Bank--Weiser space \eqref{eq:bwspace} is spanned by
the edge bubbles, so an edge of length $t\,h$ carries a basis function whose
gradient blows up like $(t h)^{-1}$, and the local problem \eqref{eq:bw} is
solved in a space that is badly scaled against the residual it must represent.
We offer this as the plausible mechanism rather than as a proof; what the
experiment establishes is the fact, and the fact is that the relevant geometric
parameter for Bank--Weiser is not the one the quality-control tool reports.

\paragraph{Anisotropy and accumulated damage cost it too.}
Over the anisotropic family, $I^{\mathrm{BW}}$ grows from $1.017$ to $2.297$
against $1.032$ to $1.401$ for the equilibrated flux: the two estimators start
with Bank--Weiser ahead and end with it $60\%$ behind
(Figure~\ref{fig:effectivity}(c)). The same crossing happens as damage
accumulates in family \ref{it:F4}, Table~\ref{tab:damaged}: at $m=64$ in a
contiguous block, $I^{\mathrm{BW}} = 1.321$ against $I^{\mathrm{eq}} = 1.028$.

\paragraph{The diffusion contrast breaks the unweighted variant.}
The largest entry of the campaign is the last row of Table~\ref{tab:cases}:
with $\mu$ jumping by $10^4$ across a mesh-aligned interface,
$I^{\mathrm{BW}} = 12.514$ where $I^{\mathrm{eq}} = 1.032$. This is remark~(c)
of Section~\ref{sec:comparison} made quantitative. The facet term of
\eqref{eq:bw} carries $\mu = 10^4$ from the high-coefficient side, and on the
low-coefficient cell it is weighed against $\mu = 1$ in the local bilinear
form; the estimator therefore attributes to that cell an error four orders of
magnitude too large.

Coefficient-weighted averages are the known cure, and we have measured
them, so as not to charge the estimator with a defect of one particular
variant. Replacing the $1/2$ split of \eqref{eq:bw} by the weights
$\omega_K = \mu_K/(\mu_+ + \mu_-)$ --- which give the low-coefficient side a
share of the jump proportional to its own $\mu$, and restore the
$\mu_{\max}^{-1/2}$ scaling of the residual estimator \eqref{eq:residual} ---
brings the index from $12.514$ down to $1.014$, and changes nothing at all on
the runs with constant $\mu$, where the two variants agree to machine
precision. The column $I^{\mathrm{BW}_\mu}$ of Table~\ref{tab:cases} reports
it. The $12.5$ is therefore entirely an artefact of the unweighted split, and
the statement that survives is the one about \emph{who} has to know: the
contrast-robustness of the equilibrated flux estimator is a theorem
\michel{\cite{caizhang2012}} that requires nothing of the user, whereas that of
Bank--Weiser is a modelling choice that has to be got right, silently, before
the computation.

\paragraph{It does underestimate, and nothing signals it.}
Across the $32$ synthetic runs on the unit square and the one real mesh with a
known solution, $I^{\mathrm{BW}}$ stayed above one, with a minimum of $1.004$;
Bank--Weiser should therefore \emph{not} be lumped together with
Zienkiewicz--Zhu, which was below one every single time. On the
L-shaped domain, however, it is below one on all four meshes, down to
$I^{\mathrm{BW}} = 0.9935$ (Table~\ref{tab:lshape}). The deficit is small,
under a percent, and a practitioner would not notice it. The
saturation assumption on which reliability rests is not satisfied there,
nothing in the computed output says so, and the estimator returns a number that
is confidently below the truth. The difference with the equilibrated flux is
not that the latter is sharper, but that its failure mode is a certificate one
can check (step~4 of Algorithm~\ref{alg:main}) rather than an assumption one
must believe.

\paragraph{Cost.}
On the finest structured mesh ($32\,768$ cells), the primal solve takes
$\michel{0.13}$~s, the Bank--Weiser estimator $\michel{0.27}$~s and the
\emph{global} mixed reconstruction $\michel{1.10}$~s: Bank--Weiser costs about
twice the primal solve, the global equilibrated reconstruction about eight
times. We deliberately do not quote a wall-clock figure for the
patchwise reconstruction. Its $16\,641$ local systems have at most $43$
unknowns each, so its arithmetic content is well below one second of dense
linear algebra; but our driver loops over the patches in \texttt{Python} and
spends $\michel{5.5}$~s doing so, which measures the interpreter and not the method. The
meaningful statement is the algebraic one: per vertex, one dense symmetric
system of size at most $43$, all of them independent, against one global
saddle-point solve for the reconstruction of Section~\ref{sec:mixed} and one
$3\times 3$ system per cell for Bank--Weiser.

\subsection{Boundary layers and heterogeneous coefficients}
\label{sec:numerics_other}

Table~\ref{tab:cases} collects the other two test cases on the structured
$32\times 32$ mesh.

\begin{table}[ht]
\centering
\begin{tabular}{llrrrrrr}
\toprule
case & & $\enorm{u-u_h}$ & $I^{\mathrm{eq}}$ & $I^{\mathrm{pw}}$ & $I^{\mathrm{BW}}$ & $I^{\mathrm{BW}_\mu}$ & $I^{\mathrm{ZZ}}$ \\
\midrule
\ref{it:T1} & smooth, $\mu=1$ & 1.090e-01 & 1.032 & 1.048 & 1.017 & 1.017 & 0.993 \\
\ref{it:T3} & layer, $\varepsilon=10^{-2}$ & 1.178e+00 & 2.402 & 1.583 & 1.743 & 1.743 & 0.414 \\
\ref{it:T4} & $\mu$ contrast $10^{4}$ & 7.705e+00 & 1.032 & 1.048 & 12.514 & 1.014 & 0.993 \\
\bottomrule
\end{tabular}
\caption{Test cases \ref{it:T1}, \ref{it:T3}, \ref{it:T4} on the structured
mesh \ref{it:F1} with $N=32$, $k=1$.}
\label{tab:cases}
\end{table}

\paragraph{Boundary layer.}
With $\varepsilon = 10^{-2}$ and $h = 1/32$, the layers are heavily
under-resolved and the discrete solution carries an error of the order of
the solution itself. The equilibrated estimator reports it: $I^{\mathrm{eq}} = 2.40$, i.e.\ it announces an
error at most $2.4$ times the true one --- pessimistic, but the announcement is
of the right order and it is a \emph{guarantee}. The recovery estimator
announces $I^{\mathrm{ZZ}} = 0.41$, i.e.\ it reports \emph{two and a half times
less error than there is}. On a pre-asymptotic mesh, where the user most needs
to be warned, the popular estimator is the one that gives false confidence.
This is, in our view, the most consequential single figure of the
campaign. Bank--Weiser, for its part, is the sharpest of the four here,
$I^{\mathrm{BW}} = 1.74$, and stays above one --- on this test case the
saturation assumption evidently holds.

\paragraph{Heterogeneous coefficient.}
With a contrast of $10^4$ across a mesh-aligned interface, the effectivity
indices of the equilibrated flux, of the recovery and of the residual estimator
are unchanged from the homogeneous case ($1.032$ against $1.032$, $0.993$
against $0.993$, and $5.632$ against $5.643$) --- but \emph{not} \michel{that of} Bank--Weiser, which
jumps from $1.017$ to $12.514$; its coefficient-weighted variant
returns $1.014$. See Section~\ref{sec:numerics_bw}. This confirms numerically
the $\mu$-robustness of Theorem~\ref{thm:upper}: the coefficient enters only
through the weights $\mu^{\pm 1/2}$, and no constant sees the contrast.
The patchwise reconstruction behaves identically,
$I^{\mathrm{pw}} = 1.048$ against $1.048$.

\michel{Assumption \ref{it:M3} is not decorative either, but its failure is
not one that the certificate can see. Repeating the run with the damaged
vertex placed \emph{on} the interface, so that it is dragged across it and
$\mu$ ceases to be constant on the three cells it drags along
(\texttt{interface\_check.py} in the repository), the equilibration defect
stays at $3.7\cdot 10^{-11}$: the mixed problem \eqref{eq:mixedpb}
equilibrates $\sig_h$ against whatever $f$ it is given, and \eqref{eq:EF}
holds whatever $\mu$ does inside a cell. What \ref{it:M3} controls is the
residual indicator. Its proof inserts $\mu_K^{-1/2}$, which on a crossed cell
has to be replaced by $(\min_K \mu)^{-1/2}$ --- the bound survives --- and,
more importantly, $f - \Pi_0 f$ is no longer of higher order there, since $f$
jumps by a factor $10^4$ inside the cell. The discrete solution itself is
barely affected, $\enorm{u-u_h} = 7.73$ against $7.71$, but the ratio
$\norm{\eta^{\mathrm{R}}}/\norm{\eta^{\mathrm{F}}}$ rises from $0.04$ to
above $1$, reaching $4.4$ on the crossed cells, and the guaranteed index
becomes $I^{\mathrm{eq}} = 9.2$. The bound is still a bound, and it still
holds by virtue of Theorem~\ref{thm:upper} once the weight is corrected, but
it has lost its sharpness because of three cells out of $2048$; the signal is
in the oscillation column of the output, not in the certificate. Whether
$\mu$ is constant on every cell is, in any case, a property of the data that
can be checked by inspection before the computation, which is why we do not
count \ref{it:M3} among the hypotheses the certificate is responsible for.}

\subsection{The L-shaped domain: dropping a hypothesis on purpose}
\label{sec:numerics_lshape}

Remark~\ref{rem:nonhomog} identifies the one hypothesis of
Theorem~\ref{thm:upper} that we cannot enforce in practice: that the discrete
boundary values coincide with $g$. Test case \ref{it:T2} on the meshes
\ref{it:F7} violates it --- $g_h$ is the Lagrange interpolant of
$r^{2/3}\sin(2\theta/3)$, so $u - u_h \notin H^1_0(\Om)$ and the theorem does
not apply. Table~\ref{tab:lshape} reports what happens.

\begin{table}[ht]
\centering
\begin{tabular}{lrrrrrrr}
\toprule
mesh & cells & $\enorm{u-u_h}$ & rate & $I^{\mathrm{eq}}$ & $I^{\mathrm{pw}}$ & $I^{\mathrm{BW}}$ & $I^{\mathrm{ZZ}}$ \\
\midrule
F7 L-shape $N=4$ & 96 & 1.898e-01 & --- & 1.110 & 1.187 & 0.9992 & 1.062 \\
F7 L-shape $N=8$ & 384 & 1.221e-01 & 0.64 & 1.105 & 1.177 & 0.9951 & 1.056 \\
F7 L-shape $N=16$ & 1536 & 7.799e-02 & 0.65 & 1.103 & 1.172 & 0.9939 & 1.053 \\
F7 L-shape $N=32$ & 6144 & 4.957e-02 & 0.65 & 1.101 & 1.169 & 0.9935 & 1.052 \\
\bottomrule
\end{tabular}
\caption{Test case \ref{it:T2} on the L-shaped meshes \ref{it:F7},
$k=1$. The observed convergence rate $0.65$ matches the expected $h^{2/3}$.
Theorem~\ref{thm:upper} does not cover these runs, the Dirichlet datum being
interpolated rather than exact; both reconstructions nevertheless stay above the
true error, whereas Bank--Weiser is below it on every mesh of the family.}
\label{tab:lshape}
\end{table}

Three observations, the first two of which qualify the guarantee
rather than support it.
\begin{enumerate}[nosep]
  \item \emph{The bound survives, but we cannot claim it.} Both
        $I^{\mathrm{eq}} = 1.101$--$1.110$ and
        $I^{\mathrm{pw}} = 1.169$--$1.187$ stay comfortably above one, and
        remarkably stable across four refinements of a singular solution. This
        is evidence that the boundary term neglected here is small on this
        family, not that it is small in general; on a coarser mesh or with a
        rougher datum it need not be. We report the numbers as an indication,
        and repeat that the certificate of step~4 of Algorithm~\ref{alg:main}
        does \emph{not} detect this failure --- it tests \eqref{eq:EF}, which
        still holds exactly here (the defect is at $10^{-13}$), and is silent
        about the boundary condition. This is the one gap in the guarantee that
        the computation cannot check.
  \item \emph{Bank--Weiser underestimates.} $I^{\mathrm{BW}}$ is
        $0.9992$, $0.9951$, $0.9939$, $0.9935$ on the four meshes: below one
        every time, and drifting further below under refinement. This is the
        only family of the campaign on which it does so, and it is a
        shape-regular one --- the difficulty is the corner singularity, not the
        geometry of the cells. It confirms, on a textbook test case, that the
        empirical safety observed in Section~\ref{sec:numerics_bw} is a property
        of the problems tested and not of the estimator.
  \item \emph{Zienkiewicz--Zhu is above one here.}
        $I^{\mathrm{ZZ}} = 1.05$, the only place in the campaign where it does
        not underestimate. Recovery estimators are not systematically unsafe;
        they are unsafe where superconvergence holds and their calibration is
        tuned to it, which is the smooth, structured regime. On a singular
        solution the recovered gradient is simply too rough, and the estimator
        errs the other way. Neither behaviour is under control.
\end{enumerate}

\subsection{A certified quantity of interest}
\label{sec:numerics_goal}

Table~\ref{tab:goal} applies Proposition~\ref{prop:goal} to the mean of
$u$ over the box $D = (1/4,1/2)^2$, on test case \ref{it:T1} where $J(u)$ is
known in closed form. The dual problem
$-\dive(\mu\grad z) = |D|^{-1}\mathbf{1}_D$ is solved with the same factorised
matrix as the primal one, and the same reconstruction is applied to it.

\begin{table}[ht]
\centering
\begin{tabular}{lrrrrrr}
\toprule
mesh & cells & $|J(u)-J(u_h)|$ & $\eta^{\mathrm{prim}}$ & $\eta^{\mathrm{dual}}$ & bound & $I^{\mathrm{goal}}$ \\
\midrule
F1 $N=8$ & 128 & 2.839e-02 & 4.878e-01 & 1.536e-01 & 7.493e-02 & 2.64 \\
F1 $N=16$ & 512 & 7.167e-03 & 2.316e-01 & 7.974e-02 & 1.847e-02 & 2.58 \\
F1 $N=32$ & 2048 & 1.796e-03 & 1.125e-01 & 4.036e-02 & 4.541e-03 & 2.53 \\
F1 $N=64$ & 8192 & 4.493e-04 & 5.540e-02 & 2.026e-02 & 1.122e-03 & 2.50 \\
F2 $\delta=0.30$ & 2048 & 2.083e-03 & 1.208e-01 & 4.951e-02 & 5.981e-03 & 2.87 \\
F3 needle $t=10^{-4}$ & 2048 & 1.788e-03 & 1.128e-01 & 4.149e-02 & 4.680e-03 & 2.62 \\
F3 cap $t=10^{-4}$ & 2048 & 1.808e-03 & 1.130e-01 & 4.109e-02 & 4.642e-03 & 2.57 \\
F5 $32\times2$ & 128 & 2.037e-01 & 1.633e+00 & 5.842e-01 & 9.539e-01 & 4.68 \\
\bottomrule
\end{tabular}
\caption{Guaranteed bound \eqref{eq:goalbound} on the quantity of
interest $J(v) = |D|^{-1}\int_D v$, $D = (1/4,1/2)^2$, test case \ref{it:T1}.
The goal error converges at $h^2$ and so does the product bound, so the
effectivity index is asymptotically constant. As for the energy bound, an
isolated degenerate cell leaves it essentially unchanged and global anisotropy
degrades it.}
\label{tab:goal}
\end{table}

The bound held in all eight runs. The effectivity index sits at
$2.50$--$2.64$ on the structured family and decreases monotonically under
refinement, which is the expected behaviour: $|J(u)-J(u_h)|$ and the product
$\eta^{\mathrm{prim}}\eta^{\mathrm{dual}}$ both converge at $h^2$, and what the
index measures is the cancellation between the primal and the dual error
distributions that the Cauchy--Schwarz step of \eqref{eq:goalbound} throws away.
A factor of $2.5$ is the price of a product bound; it is not the price of the
mesh. The mesh enters exactly as it does in the energy norm: a needle or a cap
with $s_K = 10^{-4}$ moves the index from $2.53$ to $\michel{2.62}$ and $2.57$
respectively, while the aspect-ratio-$16$ mesh takes it to $\michel{4.68}$ --- the
compounding of the degradation seen on each factor separately, which is what a
product of two bounds must do. A quantity of interest computed on a mesh
containing a degenerate cell can therefore be certified at essentially the same
price, and with the same constant, as the energy norm.

\subsection{Summary of the campaign}
\label{sec:numerics_summary}

The campaign comprises $49$ computations: $32$ runs on the unit square
covered by Theorem~\ref{thm:upper}, $4$ runs on the L-shaped domain that are not
(Section~\ref{sec:numerics_lshape}), $8$ goal-oriented runs
(Section~\ref{sec:numerics_goal}) and the $5$ real meshes of
Section~\ref{sec:numerics_real}. Over the runs covered by the theorem:
\begin{itemize}[nosep]
  \item the guaranteed bound $\eta \geq \enorm{u-u_h}$ held in \emph{every} run
        where the error is computable, the assertion of
        Listing~\ref{lst:fenicsx} never firing, for either
        reconstruction, and the goal-oriented bound \eqref{eq:goalbound} held in
        all $8$ of its runs;
  \item $I^{\mathrm{eq}} \in [1.008, 2.402]$, the upper end being reached only
        on the deliberately under-resolved boundary layer; on the smooth test
        case it never exceeded $1.401$, whatever the distortion;
  \item $I^{\mathrm{pw}} \in [1.036, 2.215]$ for the patchwise
        reconstruction at $p=k$, the upper end being reached on accumulated cap
        damage rather than on anisotropy: the two reconstructions of the same
        bound fail to be sharp on \emph{disjoint} families of meshes
        (Section~\ref{sec:numerics_pw});
  \item $I^{\mathrm{ZZ}} < 1$ in \emph{all} $32$ synthetic runs, with a worst
        case of $0.414$;
  \item $I^{\mathrm{BW}} \in [1.004, 12.514]$: the sharpest estimator on good
        meshes and the most fragile on bad ones. It stayed above one on
        the unit square, but fell to $0.9935$ on the L-shaped domain; its
        coefficient-weighted variant removes the $12.514$ --- it ranges over
        $[1.004, 2.664]$ --- but not the underestimation;
  \item $I^{\mathrm{res}}$ ranged over $[5.53, 10.43]$, a factor of almost two
        between the extreme meshes, so that no universal calibration of $C$
        exists;
  \item the relative equilibration defect never exceeded
        $\michel{7.2\cdot 10^{-10}}$ for the global reconstruction and
        $\michel{4.4\cdot 10^{-10}}$ for the patchwise one\michel{, boundary
        layer included}.
\end{itemize}
\michel{Two hypotheses were violated on purpose, and the outcome sorts the
hypotheses of Theorem~\ref{thm:upper} into three classes. Assumption
\ref{it:M3} --- a coefficient interface crossing a cell --- is a property of
the data, checkable by inspection before any computation; its violation
leaves the bound valid but raises the effectivity index from $1.03$ to
$9.2$, and the certificate does not see it. The equilibration \eqref{eq:EF}
is a property of the reconstruction, and the certificate of step~4 of
Algorithm~\ref{alg:main} checks it (the $H(\dive)$-conformity of $\sig_h$
holds by construction of the space). The exactness of the discrete Dirichlet
datum (Remark~\ref{rem:nonhomog}) is checkable by neither, and this is worth
stating plainly: it is the one hypothesis for which the computed output
cannot vouch.}

\subsection{Real geometries}
\label{sec:numerics_real}

We close with the meshes that reached us from the applications this
work originates in, family \ref{it:F6}. They come as \texttt{OFF} files and
illustrate the two situations of engineering practice.

\paragraph{A mesh that no solver will accept.}
The file \texttt{square\_good.1.off} triangulates $(-1,1)^2$ with $801$ right
triangles. Its quality figures are those of a perfectly structured mesh,
$s_K \equiv 0.707$ and $\theta_{\max} = 90$ degrees --- and it contains one cell
of exactly zero area, because two of its vertices are the same point to machine
precision. This is not an exotic pathology: the file is named
``good'', nothing in its quality report signals the defect, and every finite
element code will produce a singular system from it. Our reader (\texttt{offio.py}) reports the defect,
merges the coincident vertices and drops the degenerate cell, leaving $800$
cells. Since the domain is then exactly the square, the manufactured solution
$u = \cos(\pi x_1/2)\cos(\pi x_2/2)$ applies and the effectivity index can be
measured: $I^{\mathrm{eq}} = 1.054$ (Table~\ref{tab:real}). The
guarantee holds on a mesh that was not prepared for the experiment.

\paragraph{A mesh whose exact solution nobody knows.}
The four remaining meshes triangulate the letter ``A'' --- a non-convex domain
with thin features, sharp corners and re-entrant corners --- and come from two
different meshers (\texttt{as}, unstructured and graded; \texttt{aa}, uniform)
at two refinement levels each. No closed-form solution is available, which is
the normal state of affairs. We solve $-\Delta u = 1$ with homogeneous
Dirichlet data and report the \emph{certified} quantity that a practitioner
would actually obtain: the bound $\eta$, and its size relative to the energy of
the computed solution.

\begin{table}[htbp]
\centering
\small
\begin{tabular}{lrrrrrrrrr}
\toprule
mesh & cells & $s_{\min}$ & $\theta_{\max}$ & $\eta$ & $\eta/\enorm{u_h}$ & $I^{\mathrm{eq}}$ & $\eta^{\mathrm{BW}}/\eta$ & $\eta^{\mathrm{pw}}/\eta$ & defect \\
\midrule
as6 & 1327 & 0.23 & 143.4 & 1.551e-03 & 28.1\% & --- & 1.103 & 1.065 & $1\cdot 10^{-14}$ \\
as7 & 3262 & 0.26 & 145.3 & 1.166e-03 & 20.8\% & --- & 1.130 & 1.054 & $3\cdot 10^{-14}$ \\
aa6 & 1983 & 0.25 & 143.4 & 1.111e-03 & 19.8\% & --- & 1.035 & 1.062 & $7\cdot 10^{-14}$ \\
aa7 & 7924 & 0.26 & 145.3 & 5.979e-04 & 10.5\% & --- & 0.986 & 1.067 & $3\cdot 10^{-13}$ \\
square\_good & 800 & 0.71 & 90.0 & 1.797e-01 & 8.1\% & 1.054 & 1.047 & 0.997 & $7\cdot 10^{-14}$ \\
\bottomrule
\end{tabular}
\caption{Family \ref{it:F6}, the real meshes. The four glyph meshes solve
$-\Delta u = 1$ and no exact solution is available, so only the certified bound
is reported; \texttt{square\_good} carries a manufactured solution and its
effectivity index is measured. The penultimate column compares the
Bank--Weiser bound to the guaranteed one: on these meshes the two agree within
$13\%$. The next column does the same for the patchwise
reconstruction, which agrees with the global one within $7\%$ on all five
meshes. The last column is the relative equilibration defect of step~4 of
Algorithm~\ref{alg:main}.}
\label{tab:real}
\end{table}


Three things come out of Table~\ref{tab:real} and Figure~\ref{fig:real}.
\begin{enumerate}[nosep]
  \item \textbf{The bound is informative, not vacuous.} On \texttt{aa7} the
        energy-norm error is certified to be at most $10.5\%$ of the energy of
        the computed solution; on \texttt{as6} at most $28.1\%$. These are
        numbers an engineer can act on, obtained without knowing $u$.
  \item \textbf{It ranks the \claudio{meshers}.} At $1983$ cells, \texttt{aa6} certifies
        $19.8\%$, whereas \texttt{as7} needs $3262$ cells --- $64\%$ more --- to
        certify a slightly worse $20.8\%$. For this problem the uniform \claudio{mesher}
        is the more efficient of the two, and the guaranteed bound says so
        without any reference computation. Comparing two \claudio{meshers} objectively on
        a geometry where the solution is unknown is, to our knowledge,
        not otherwise possible without a reference computation.
  \item \textbf{It localises the error where one expects it.}
        Figure~\ref{fig:real}(a,b) maps $\eta_K$: the indicator concentrates in
        the crossbar of the glyph, where the cells are elongated and where the
        two re-entrant corners sit. Panel~(c) confirms that the two \claudio{meshers}
        differ in quality --- the \texttt{aa} family is uniformly better in
        scaled Jacobian --- but that both stay in the comfortable range
        $s_K \geq 0.23$: these real meshes are far less distorted than
        the synthetic ones of Table~\ref{tab:main}, so that what they test is
        the geometry of the domain rather than the geometry of the cells.
\end{enumerate}

\begin{figure}[htbp]
\centering
\includegraphics[width=\textwidth]{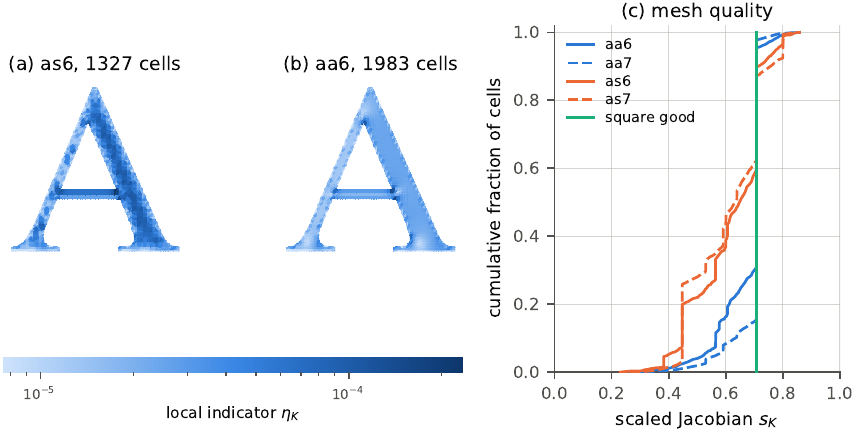}
\caption{The real meshes of family \ref{it:F6}. (a), (b) the local indicator
$\eta_K$ on the two coarse glyph meshes, on a common logarithmic colour scale;
the error concentrates in the crossbar and at the re-entrant corners. (c)
cumulative distribution of the scaled Jacobian: hue distinguishes the mesher,
dashing the refinement level.}
\label{fig:real}
\end{figure}

The patchwise reconstruction behaves on these meshes as it does on the
synthetic ones, returning a bound within $7\%$ of the global one
(Table~\ref{tab:real}) --- unsurprisingly, since by Figure~\ref{fig:real}(c)
these meshes contain no flattened cell of the kind that
Section~\ref{sec:numerics_pw} identified as its weak point. A quantity of
interest on this geometry --- say the mean of $u$ over the crossbar --- would be
certified by \eqref{eq:goalbound} at the price of one extra solve, exactly as in
Section~\ref{sec:numerics_goal}; we have not run it, since no reference value of
$J(u)$ exists here against which the resulting bound could be calibrated.

\section{Conclusion and perspectives}
\label{sec:conclusion}

We have argued, and demonstrated on $49$ numerical experiments, that
equilibrated flux estimators answer a question that practitioners actually ask
--- \emph{how wrong is my computation on this mesh?} --- in a regime where the
standard tools are silent: coarse meshes, distorted cells, no asymptotics. The
guarantee comes with a constant equal to one and, remarkably, with no geometric
assumption; the price is a possible loss of sharpness, which is the benign
failure mode.

The campaign of Section~\ref{sec:numerics} supports this rather sharply. An
isolated degenerate cell, needle or cap, with a scaled Jacobian of $5\cdot
10^{-5}$ and a maximum angle of $180$ degrees to printed precision, changes the
effectivity index of the equilibrated estimator by less than $10^{-3}$; global
anisotropy degrades it, but only from $1.03$ to $1.40$; and it never fell below
one. Over the same runs the Zienkiewicz--Zhu estimator was below one every
single time, reaching $0.41$ on an under-resolved boundary layer --- reporting
less than half of the error actually committed, in the very situation where a
warning is needed.

The comparison with Bank--Weiser is the most informative one, that
estimator being both accurate and widely used, so that the trade-off is visible
in the numbers. It is sharper than the equilibrated flux on every well-shaped mesh, by
roughly a factor of two in the overestimation, and sharper still on the
under-resolved layer. It is also the estimator that distortion punishes most: a
single needle cell out of $2048$ multiplies its index by $2.6$, global
anisotropy by $2.3$, a coefficient contrast of $10^4$ by $12$ --- each time
while the guaranteed index stays at $1.03$. Of these, the last is
removable: a coefficient-weighted facet split brings the $12$ back to $1.01$
(Section~\ref{sec:numerics_bw}). The first two are not, and the L-shaped domain
adds a failure that no calibration repairs, the estimator falling below the true
error on every mesh of that family while returning no indication that it has.
Sharpness and robustness are traded against each other here, and which one a
practitioner should buy depends on whether the mesh is one they chose.

Two findings were not anticipated when the campaign was designed, and
they are the ones we would single out. First, the two reconstructions of the
same guaranteed bound --- the global mixed one and the patchwise one --- lose
sharpness on disjoint families of meshes: the global one under global
anisotropy, the patchwise one under an accumulation of flattened cells inside a
single vertex patch. Since both are certificates, the sharper of the two may
simply be kept, and their minimum is again guaranteed. Second, no component of
the quality vector \eqref{eq:qualityvector} predicts the local effectivity
index: the largest correlation over the campaign is $\michel{0.48}$, and the sign of the
correlation with the shape-regularity ratio reverses between a needle and a cap
(Section~\ref{sec:numerics_corr}). A mesh-quality report and an error bound
answer different questions, and the second cannot be inferred from the first.

Open directions:
\begin{enumerate}[nosep]
  \item \textbf{Sharp characterisation of admissible distortion.} Following
        \cite{duprez2019damage}, find a necessary and sufficient condition on
        the damaged region for the effectivity index to stay bounded.
        Section~\ref{sec:numerics_pw} rules out the diameter of the
        damaged region as the relevant parameter and points instead to the
        number of flattened cells sharing a vertex patch; turning that
        observation into a bound on $C_{\mathrm{st}}$ is open.
  \item \michel{\textbf{Three dimensions.} Theorem~\ref{thm:upper} is written
        for $d \in \{2,3\}$ and the Payne--Weinberger constant does not depend
        on the dimension, so the guarantee transfers verbatim to tetrahedral
        meshes. The sharpness does not: in the experience of one of us with
        three-dimensional eddy-current problems, the lowest-order
        reconstruction in $\RT_0$ leaves the effectivity index well above one,
        and the $\RT_1$ reconstruction that brings it back is expensive.
        Quantifying this trade-off on distorted tetrahedra --- slivers, for
        which no two-dimensional analogue exists --- is the natural
        continuation of the present campaign, and the anisotropic
        interpolation theory of \cite{acosta2011} suggests it is not
        hopeless.}
  \item \textbf{Adaptivity on bad meshes.} The convergence and optimality theory
        of AFEM \cite{dorfler1996,cascon2008,nochetto2009} requires shape
        regularity; does D\"orfler marking still behave well in practice when
        the initial mesh is distorted and cannot be improved?
  \item \textbf{Beyond Poisson.} Elasticity (with the locking issue),
        Helmholtz \cite{chaumontfrelet2021}, eddy currents and magnetodynamics
        \cite{tang2013,creuse2017maxwell}; also the linearisation and algebraic
        error components in the spirit of \cite{ernvohralik2013}.
  \item \textbf{Extreme deformations.} Estimators in the X-MESH setting
        \cite{moes2023}, where elements are deliberately degenerated.
  \item \textbf{Meshless and immersed methods.} Since no mesh enters the proof
        of Theorem~\ref{thm:upper}, the framework should transfer.
  \item \textbf{Superconvergent postprocessing} \cite{muga2023} as a cheaper
        route to an equilibrated flux.
  \item \textbf{Non-homogeneous Dirichlet data.} The one hypothesis of
        Theorem~\ref{thm:upper} that \michel{neither the data nor the
        certificate can vouch for}
        (Remark~\ref{rem:nonhomog}). A computable bound on the boundary term,
        even a pessimistic one, would close the last gap between what is proved
        and what is run.
\end{enumerate}

\paragraph{Code and data availability.}
All the scripts, the input meshes and the raw result files needed to
regenerate every table and every figure of this paper are openly available at
\begin{center}
\url{https://github.com/michelduprez/A-posteriori-error-estimation-for-coarse-or-distorted-meshes}
\end{center}
The campaign runs in serial with \texttt{FEniCSx} (\texttt{dolfinx}~0.9.0); the
repository \texttt{README} documents the role of each file and the exact
commands used.

\paragraph{Acknowledgements.}
This work was partially supported by the regional programme
MATH-AmSud through the project NeuralFEM (2026--2027), which brings together
teams from France, Uruguay, Chile and Colombia. This research is partly funded by la Fédération de Recherche d'Animation Bézout CNRS.
We thank 
Roland Becker,
Marek Bucki,
St\'ephane Bordas,
Daniela Capatina,
Alexei Lozinski,
Robert Luce,
Pascal Omnes,
Yohan Payan,
Martin Vohralik
for highly stimulating discussions around this topic. 

\bibliographystyle{alpha}
\bibliography{biblio}

@article {muga2023,
    AUTHOR = {Muga, Ignacio and Rojas, Sergio and Vega, Patrick},
     TITLE = {An adaptive superconvergent mixed finite element method based
              on local residual minimization},
   JOURNAL = {SIAM J. Numer. Anal.},
  FJOURNAL = {SIAM Journal on Numerical Analysis},
    VOLUME = {61},
      YEAR = {2023},
    NUMBER = {5},
     PAGES = {2084--2105},
      ISSN = {0036-1429,1095-7170},
   MRCLASS = {65N30},
  MRNUMBER = {4641648},
       DOI = {10.1137/22M1526307},
       URL = {https://doi.org/10.1137/22M1526307},
}

@InCollection{carstensen2013aspects,
 Author = {Carstensen, C. and Merdon, C. and Neumann, Johannes},
 Title = {Aspects of guaranteed error control in {CPDEs}},
 BookTitle = {Numerical solution of partial differential equations: Theory, algorithms, and their applications. Papers based on the international symposium, Sozopol, Bulgaria, June 7--8, 2013. In Honor of Professor Raytcho Lazarov's 40 years of research in computational methods and applied mathematics and on the occasion of his 70th birthday},
 ISBN = {978-1-4614-7171-4; 978-1-4614-7172-1},
 Pages = {103--119},
 Year = {2013},
 Publisher = {New York, NY: Springer},
 Language = {English},
 DOI = {10.1007/978-1-4614-7172-1_6},
 zbMATH = {6207567},
 Zbl = {1275.65075}
}

@Article{carstensen2013,
 Author = {Carstensen, C. and Merdon, C.},
 Title = {Effective postprocessing for equilibration a posteriori error estimators},
 FJournal = {Numerische Mathematik},
 Journal = {Numer. Math.},
 ISSN = {0029-599X},
 Volume = {123},
 Number = {3},
 Pages = {425--459},
 Year = {2013},
 Language = {English},
 DOI = {10.1007/s00211-012-0494-4},
 zbMATH = {6141194},
 Zbl = {1267.65162}
}

@Article{veeser2009,
 Author = {Veeser, Andreas and Verf{\"u}rth, R{\"u}diger},
 Title = {Explicit upper bounds for dual norms of residuals},
 FJournal = {SIAM Journal on Numerical Analysis},
 Journal = {SIAM J. Numer. Anal.},
 ISSN = {0036-1429},
 Volume = {47},
 Number = {3},
 Pages = {2387--2405},
 Year = {2009},
 Language = {English},
 DOI = {10.1137/080738283},
 zbMATH = {5736099},
 Zbl = {1195.65158}
}

@Article{carstensen2014,
 Author = {Carstensen, C. and Merdon, C.},
 Title = {Refined fully explicit a posteriori residual-based error control},
 FJournal = {SIAM Journal on Numerical Analysis},
 Journal = {SIAM J. Numer. Anal.},
 ISSN = {0036-1429},
 Volume = {52},
 Number = {4},
 Pages = {1709--1728},
 Year = {2014},
 Language = {English},
 DOI = {10.1137/120896517},
 URL = {semanticscholar.org/paper/9460e8373bbeb8a08b3ecae49548ba4a105cdf24},
 zbMATH = {6365777},
 Zbl = {1309.65124}
}

@article {duprez2019damage,
    AUTHOR = {Duprez, Michel and Lleras, Vanessa and Lozinski, Alexei},
     TITLE = {Finite element method with local damage of the mesh},
   JOURNAL = {ESAIM Math. Model. Numer. Anal.},
  FJOURNAL = {ESAIM. Mathematical Modelling and Numerical Analysis},
    VOLUME = {53},
      YEAR = {2019},
    NUMBER = {6},
     PAGES = {1871--1891},
      ISSN = {2822-7840,2804-7214},
   MRCLASS = {65N30 (65N12 65N15)},
  MRNUMBER = {4019760},
MRREVIEWER = {Petr\ Knobloch},
       DOI = {10.1051/m2an/2019023},
       URL = {https://doi.org/10.1051/m2an/2019023},
}

@article {moes2023,
    AUTHOR = {Mo\"{e}s, Nicolas and Remacle, Jean-Fran\c{c}ois and
              Lambrechts, Jonathan and L\'{e}, Beno\^{i}t and Chevaugeon,
              Nicolas},
     TITLE = {The e{X}treme mesh deformation approach ({X}-{MESH}) for the
              {S}tefan phase change model},
   JOURNAL = {J. Comput. Phys.},
  FJOURNAL = {Journal of Computational Physics},
    VOLUME = {477},
      YEAR = {2023},
     PAGES = {Paper No. 111878, 20},
      ISSN = {0021-9991,1090-2716},
   MRCLASS = {65M50 (80A22)},
  MRNUMBER = {4541860},
       DOI = {10.1016/j.jcp.2022.111878},
       URL = {https://doi.org/10.1016/j.jcp.2022.111878},
}

@misc{vohralik_notes,
  author       = {Vohral{\'\i}k, Martin},
  title        = {A Posteriori Error Estimates for Efficiency and Error Control
                  in Numerical Simulations},
  howpublished = {Lecture notes, Course NPMA041, Charles University Prague /
                  Sorbonne Universit{\'e}},
  year         = {2020},
  url          = {https://who.rocq.inria.fr/Martin.Vohralik/Enseig/APost/a_posteriori.pdf},
}

@article{pragersynge1947,
  author  = {Prager, W. and Synge, J. L.},
  title   = {Approximations in elasticity based on the concept of function space},
  journal = {Quart. Appl. Math.},
  volume  = {5},
  year    = {1947},
  pages   = {241--269},
  doi     = {10.1090/qam/25902},
}

@article{braess2008,
  author  = {Braess, Dietrich and Sch{\"o}berl, Joachim},
  title   = {Equilibrated residual error estimator for edge elements},
  journal = {Math. Comp.},
  volume  = {77},
  number  = {262},
  year    = {2008},
  pages   = {651--672},
  doi     = {10.1090/S0025-5718-07-02080-7},
}

@article{braess2009,
  author  = {Braess, Dietrich and Pillwein, Veronika and Sch{\"o}berl, Joachim},
  title   = {Equilibrated residual error estimates are p-robust},
  journal = {Comput. Methods Appl. Mech. Engrg.},
  volume  = {198},
  number  = {13--14},
  year    = {2009},
  pages   = {1189--1197},
  doi     = {10.1016/j.cma.2008.12.010},
}

@article{ernvohralik2015,
  author  = {Ern, Alexandre and Vohral{\'\i}k, Martin},
  title   = {Polynomial-degree-robust a posteriori estimates in a unified setting
             for conforming, nonconforming, discontinuous {G}alerkin, and mixed
             discretizations},
  journal = {SIAM J. Numer. Anal.},
  volume  = {53},
  number  = {2},
  year    = {2015},
  pages   = {1058--1081},
  doi     = {10.1137/130950100},
}

@article{ernvohralik2020,
  author  = {Ern, Alexandre and Vohral{\'\i}k, Martin},
  title   = {Stable broken {$H^1$} and {$H(\mathrm{div})$} polynomial extensions for
             polynomial-degree-robust potential and flux reconstruction in three
             space dimensions},
  journal = {Math. Comp.},
  volume  = {89},
  number  = {322},
  year    = {2020},
  pages   = {551--594},
  doi     = {10.1090/mcom/3482},
}

@article{destuynder1999,
  author  = {Destuynder, Philippe and M{\'e}tivet, Brigitte},
  title   = {Explicit error bounds in a conforming finite element method},
  journal = {Math. Comp.},
  volume  = {68},
  number  = {228},
  year    = {1999},
  pages   = {1379--1396},
  doi     = {10.1090/S0025-5718-99-01093-5},
}

@book{ladeveze2004,
  author    = {Lad{\`e}v{\`e}ze, Pierre and Pelle, Jean-Pierre},
  title     = {Mastering Calculations in Linear and Nonlinear Mechanics},
  publisher = {Springer},
  year      = {2004},
}

@article{luce2004,
  author  = {Luce, Robert and Wohlmuth, Barbara I.},
  title   = {A local a posteriori error estimator based on equilibrated fluxes},
  journal = {SIAM J. Numer. Anal.},
  volume  = {42},
  number  = {4},
  year    = {2004},
  pages   = {1394--1414},
  doi     = {10.1137/S0036142903433790},
}

@article{becker2016,
  author  = {Becker, Roland and Capatina, Daniela and Luce, Robert},
  title   = {Local flux reconstructions for standard finite element methods on
             triangular meshes},
  journal = {SIAM J. Numer. Anal.},
  volume  = {54},
  number  = {4},
  year    = {2016},
  pages   = {2684--2706},
  doi     = {10.1137/16M1064817},
}

@article{creuse2017,
  author  = {Creus{\'e}, Emmanuel and Nicaise, Serge and Tang, Zuqi},
  title   = {Goal-oriented error estimation based on equilibrated flux and potential
             reconstruction for the approximation of elliptic and parabolic problems},
  journal = {Comput. Math. Appl.},
  volume  = {146},
  year    = {2023},
  pages   = {323--338},
  doi     = {10.1016/j.camwa.2023.07.008},
}

@article{repin2000,
  author  = {Repin, Sergey I.},
  title   = {A posteriori error estimation for variational problems with uniformly
             convex functionals},
  journal = {Math. Comp.},
  volume  = {69},
  number  = {230},
  year    = {2000},
  pages   = {481--500},
  doi     = {10.1090/S0025-5718-99-01190-4},
}

@book{han2005,
  author    = {Han, Weimin},
  title     = {A Posteriori Error Analysis via Duality Theory. With Applications
               in Modeling and Numerical Approximations},
  series    = {Advances in Mechanics and Mathematics},
  volume    = {8},
  publisher = {Springer},
  address   = {New York},
  year      = {2005},
  doi       = {10.1007/b101775},
}

@article{bartels2023,
  author  = {Bartels, S{\"o}ren and Kaltenbach, Alex},
  title   = {Explicit and efficient error estimation for convex minimization
             problems},
  journal = {Math. Comp.},
  volume  = {92},
  number  = {343},
  year    = {2023},
  pages   = {2247--2279},
  doi     = {10.1090/mcom/3821},
}

@book{repin2008,
  author    = {Repin, Sergey I.},
  title     = {A Posteriori Estimates for Partial Differential Equations},
  publisher = {Walter de Gruyter},
  address   = {Berlin},
  year      = {2008},
}

@book{verfurth2013,
  author    = {Verf{\"u}rth, R{\"u}diger},
  title     = {A Posteriori Error Estimation Techniques for Finite Element Methods},
  publisher = {Oxford University Press},
  year      = {2013},
}

@book{ainsworthoden2000,
  author    = {Ainsworth, Mark and Oden, J. Tinsley},
  title     = {A Posteriori Error Estimation in Finite Element Analysis},
  publisher = {Wiley},
  year      = {2000},
}

@article{zz1987,
  author  = {Zienkiewicz, O. C. and Zhu, J. Z.},
  title   = {A simple error estimator and adaptive procedure for practical
             engineering analysis},
  journal = {Internat. J. Numer. Methods Engrg.},
  volume  = {24},
  number  = {2},
  year    = {1987},
  pages   = {337--357},
  doi     = {10.1002/nme.1620240206},
}

@article{bankweiser1985,
  author  = {Bank, Randolph E. and Weiser, Alan},
  title   = {Some a posteriori error estimators for elliptic partial differential
             equations},
  journal = {Math. Comp.},
  volume  = {44},
  number  = {170},
  year    = {1985},
  pages   = {283--301},
  doi     = {10.1090/S0025-5718-1985-0777265-X},
}

@article{bulle2023,
  author  = {Bulle, Rapha{\"e}l and Hale, Jack S. and Lozinski, Alexei and
             Bordas, St{\'e}phane P. A. and Chouly, Franz},
  title   = {Hierarchical a posteriori error estimation of {B}ank--{W}eiser type in
             the {FEniCS} {P}roject},
  journal = {Comput. Math. Appl.},
  volume  = {131},
  year    = {2023},
  pages   = {103--123},
  doi     = {10.1016/j.camwa.2022.12.003},
}

@article{payneweinberger1960,
  author  = {Payne, L. E. and Weinberger, H. F.},
  title   = {An optimal {P}oincar{\'e} inequality for convex domains},
  journal = {Arch. Rational Mech. Anal.},
  volume  = {5},
  year    = {1960},
  pages   = {286--292},
  doi     = {10.1007/BF00252910},
}

@article{knupp2001,
  author  = {Knupp, Patrick M.},
  title   = {Algebraic mesh quality metrics},
  journal = {SIAM J. Sci. Comput.},
  volume  = {23},
  number  = {1},
  year    = {2001},
  pages   = {193--218},
  doi     = {10.1137/S1064827500371499},
}

@article{babuskaaziz1976,
  author  = {Babu{\v{s}}ka, Ivo and Aziz, A. K.},
  title   = {On the angle condition in the finite element method},
  journal = {SIAM J. Numer. Anal.},
  volume  = {13},
  number  = {2},
  year    = {1976},
  pages   = {214--226},
  doi     = {10.1137/0713021},
}

@article{krizek1992,
  author  = {K{\v{r}}{\'\i}{\v{z}}ek, Michal},
  title   = {On the maximum angle condition for linear tetrahedral elements},
  journal = {SIAM J. Numer. Anal.},
  volume  = {29},
  number  = {2},
  year    = {1992},
  pages   = {513--520},
  doi     = {10.1137/0729031},
}

@article{apel1999,
  author    = {Apel, Thomas},
  title     = {Anisotropic Finite Elements: Local Estimates and Applications},
  publisher = {Teubner},
  year      = {1999},
  journal   = {Advances in Numerical Mathematics},
}

@article{formaggia2001,
  author  = {Formaggia, Luca and Perotto, Simona},
  title   = {New anisotropic a priori error estimates},
  journal = {Numer. Math.},
  volume  = {89},
  number  = {4},
  year    = {2001},
  pages   = {641--667},
  doi     = {10.1007/s002110100273},
}

@article{picasso2003,
  author  = {Picasso, Marco},
  title   = {An anisotropic error indicator based on {Z}ienkiewicz--{Z}hu error
             estimator: application to elliptic and parabolic problems},
  journal = {SIAM J. Sci. Comput.},
  volume  = {24},
  number  = {4},
  year    = {2003},
  pages   = {1328--1355},
  doi     = {10.1137/S1064827501398578},
}

@article{shewchuk2002,
  author  = {Shewchuk, Jonathan Richard},
  title   = {What is a good linear finite element? {I}nterpolation, conditioning,
             anisotropy, and quality measures},
  journal = {Preprint, University of California at Berkeley},
  year    = {2002},
}

@article{bucki2010,
  author  = {Bucki, Marek and Lobos, Claudio and Payan, Yohan},
  title   = {A fast and robust patient specific finite element mesh registration
             technique: Application to 60 clinical cases},
  journal = {Med. Image Anal.},
  volume  = {14},
  number  = {3},
  year    = {2010},
  pages   = {303--317},
  doi     = {10.1016/j.media.2010.02.003},
}

@article{lobos2013,
  author  = {Lobos, Claudio},
  title   = {A set of mixed-elements patterns for domain boundary approximation in
             hexahedral meshes},
  journal = {Stud. Health Technol. Inform.},
  volume  = {184},
  year    = {2013},
  pages   = {268--274},
}

@article{chaumontfrelet2021,
  author  = {Chaumont-Frelet, Th{\'e}ophile and Ern, Alexandre and Vohral{\'\i}k, Martin},
  title   = {On the derivation of guaranteed and $p$-robust a posteriori error
             estimates for the {H}elmholtz equation},
  journal = {Numer. Math.},
  volume  = {148},
  number  = {3},
  year    = {2021},
  pages   = {525--573},
  doi     = {10.1007/s00211-021-01192-w},
}

@book{ernguermond2021,
  author    = {Ern, Alexandre and Guermond, Jean-Luc},
  title     = {Finite Elements {I}: Approximation and Interpolation},
  publisher = {Springer},
  year      = {2021},
}

@misc{blechta_dolfintape,
  author       = {Blechta, Jan},
  title        = {{dolfin-tape}: {DOLFIN} tools for a posteriori error
                  estimation},
  howpublished = {Software repository},
  year         = {2016},
  url          = {https://github.com/blechta/dolfin-tape},
}

@article{babuskarheinboldt1978,
  author  = {Babu{\v{s}}ka, Ivo and Rheinboldt, Werner C.},
  title   = {Error estimates for adaptive finite element computations},
  journal = {SIAM J. Numer. Anal.},
  volume  = {15},
  number  = {4},
  year    = {1978},
  pages   = {736--754},
  doi     = {10.1137/0715049},
}

@article{ladeveze1983,
  author  = {Lad{\`e}v{\`e}ze, Pierre and Leguillon, Dominique},
  title   = {Error estimate procedure in the finite element method and
             applications},
  journal = {SIAM J. Numer. Anal.},
  volume  = {20},
  number  = {3},
  year    = {1983},
  pages   = {485--509},
  doi     = {10.1137/0720033},
}

@book{braess2007,
  author    = {Braess, Dietrich},
  title     = {Finite Elements: Theory, Fast Solvers and Applications in Solid
               Mechanics},
  edition   = {3},
  publisher = {Cambridge University Press},
  year      = {2007},
}

@incollection{raviartthomas1977,
  author    = {Raviart, P.-A. and Thomas, J. M.},
  title     = {A mixed finite element method for 2nd order elliptic problems},
  booktitle = {Mathematical Aspects of Finite Element Methods},
  series    = {Lecture Notes in Math.},
  volume    = {606},
  publisher = {Springer},
  year      = {1977},
  pages     = {292--315},
  doi       = {10.1007/BFb0064470},
}

@article{bdm1985,
  author  = {Brezzi, Franco and Douglas, Jim and Marini, L. Donatella},
  title   = {Two families of mixed finite elements for second order elliptic
             problems},
  journal = {Numer. Math.},
  volume  = {47},
  number  = {2},
  year    = {1985},
  pages   = {217--235},
  doi     = {10.1007/BF01389710},
}

@book{boffibrezzifortin2013,
  author    = {Boffi, Daniele and Brezzi, Franco and Fortin, Michel},
  title     = {Mixed Finite Element Methods and Applications},
  series    = {Springer Series in Computational Mathematics},
  volume    = {44},
  publisher = {Springer},
  year      = {2013},
  doi       = {10.1007/978-3-642-36519-5},
}

@article{clement1975,
  author  = {Cl{\'e}ment, Philippe},
  title   = {Approximation by finite element functions using local
             regularization},
  journal = {Rev. Fran{\c{c}}aise Automat. Informat. Recherche Op{\'e}rationnelle
             S{\'e}r. Rouge Anal. Num{\'e}r.},
  volume  = {9},
  number  = {R-2},
  year    = {1975},
  pages   = {77--84},
}

@article{scottzhang1990,
  author  = {Scott, L. Ridgway and Zhang, Shangyou},
  title   = {Finite element interpolation of nonsmooth functions satisfying
             boundary conditions},
  journal = {Math. Comp.},
  volume  = {54},
  number  = {190},
  year    = {1990},
  pages   = {483--493},
  doi     = {10.1090/S0025-5718-1990-1011446-7},
}

@article{zz1992a,
  author  = {Zienkiewicz, O. C. and Zhu, J. Z.},
  title   = {The superconvergent patch recovery and a posteriori error estimates.
             {P}art 1: {T}he recovery technique},
  journal = {Internat. J. Numer. Methods Engrg.},
  volume  = {33},
  number  = {7},
  year    = {1992},
  pages   = {1331--1364},
  doi     = {10.1002/nme.1620330702},
}

@article{zz1992b,
  author  = {Zienkiewicz, O. C. and Zhu, J. Z.},
  title   = {The superconvergent patch recovery and a posteriori error estimates.
             {P}art 2: {E}rror estimates and adaptivity},
  journal = {Internat. J. Numer. Methods Engrg.},
  volume  = {33},
  number  = {7},
  year    = {1992},
  pages   = {1365--1382},
  doi     = {10.1002/nme.1620330703},
}

@article{dorfler1996,
  author  = {D{\"o}rfler, Willy},
  title   = {A convergent adaptive algorithm for {P}oisson's equation},
  journal = {SIAM J. Numer. Anal.},
  volume  = {33},
  number  = {3},
  year    = {1996},
  pages   = {1106--1124},
  doi     = {10.1137/0733054},
}

@article{cascon2008,
  author  = {Casc{\'o}n, J. Manuel and Kreuzer, Christian and
             Nochetto, Ricardo H. and Siebert, Kunibert G.},
  title   = {Quasi-optimal convergence rate for an adaptive finite element
             method},
  journal = {SIAM J. Numer. Anal.},
  volume  = {46},
  number  = {5},
  year    = {2008},
  pages   = {2524--2550},
  doi     = {10.1137/07069047X},
}

@incollection{nochetto2009,
  author    = {Nochetto, Ricardo H. and Siebert, Kunibert G. and Veeser, Andreas},
  title     = {Theory of adaptive finite element methods: an introduction},
  booktitle = {Multiscale, Nonlinear and Adaptive Approximation},
  publisher = {Springer},
  year      = {2009},
  pages     = {409--542},
  doi       = {10.1007/978-3-642-03413-8_12},
}

@article{ernvohralik2013,
  author  = {Ern, Alexandre and Vohral{\'\i}k, Martin},
  title   = {Adaptive inexact {N}ewton methods with a posteriori stopping
             criteria for nonlinear diffusion {PDE}s},
  journal = {SIAM J. Sci. Comput.},
  volume  = {35},
  number  = {4},
  year    = {2013},
  pages   = {A1761--A1791},
  doi     = {10.1137/120896918},
}

@book{freygeorge2008,
  author    = {Frey, Pascal J. and George, Paul-Louis},
  title     = {Mesh Generation: Application to Finite Elements},
  edition   = {2},
  publisher = {ISTE/Wiley},
  year      = {2008},
}

@article{acosta2011,
  author  = {Acosta, Gabriel and Apel, Thomas and Dur{\'a}n, Ricardo G. and
             Lombardi, Ariel L.},
  title   = {Error estimates for {R}aviart--{T}homas interpolation of any order
             on anisotropic tetrahedra},
  journal = {Math. Comp.},
  volume  = {80},
  number  = {273},
  year    = {2011},
  pages   = {141--163},
  doi     = {10.1090/S0025-5718-2010-02406-8},
}

@article{tang2013,
  author  = {Tang, Zuqi and Le Menach, Yvonnick and Creus{\'e}, Emmanuel and
             Nicaise, Serge and Piriou, Francis and Siauve, Noel},
  title   = {Residual based a posteriori error estimators for harmonic
             {$A/\varphi$} and {$T/\Omega$} formulations in eddy current
             problems},
  journal = {IEEE Trans. Magn.},
  volume  = {49},
  number  = {5},
  year    = {2013},
  pages   = {1721--1724},
  doi     = {10.1109/TMAG.2013.2241412},
}

@article{creuse2017maxwell,
  author  = {Creus{\'e}, Emmanuel and Nicaise, Serge and Tittarelli, Roberta},
  title   = {A guaranteed equilibrated error estimator for the
             {$\mathbf{A}$--$\varphi$} and {$\mathbf{T}$--$\Omega$}
             magnetodynamic harmonic formulations of the {M}axwell system},
  journal = {IMA J. Numer. Anal.},
  volume  = {37},
  number  = {2},
  year    = {2017},
  pages   = {750--773},
  doi     = {10.1093/imanum/drw026},
}

@incollection{bernardi2001,
  author    = {Bernardi, Christine and M{\'e}tivet, Brigitte and
               Verf{\"u}rth, R{\"u}diger},
  title     = {Analyse num{\'e}rique d'indicateurs d'erreur},
  booktitle = {Maillage et adaptation},
  editor    = {George, Paul-Louis},
  publisher = {Herm{\`e}s},
  address   = {Paris},
  year      = {2001},
  note      = {TO BE CHECKED},
}

@article{caizhang2012,
  author  = {Cai, Zhiqiang and Zhang, Shun},
  title   = {Robust equilibrated residual error estimator for diffusion
             problems: conforming elements},
  journal = {SIAM J. Numer. Anal.},
  volume  = {50},
  number  = {1},
  year    = {2012},
  pages   = {151--170},
  doi     = {10.1137/100803857},
}



\end{document}